\documentclass[10pt]{article}
\usepackage{amsmath, amssymb,amsxtra}
\newtheorem{thm}{{\sc Theorem}}[section]
\newtheorem{prop}[thm]{{\sc Proposition}}
\newtheorem{cor}[thm]{{\sc Corollary}}
\newtheorem{lem}[thm]{{\sc Lemma}}
\newtheorem{df}[thm]{{\sc Definition}}

\newenvironment{proof}{\begin{sc}\noindent Proof: \end{sc}}{
     \hbox to 2em{}\nobreak\hfill$\blacksquare$\par\medskip}

\newcommand{\LR}{\hbox{Little\-wood-Richard\-son}}

\catcode`\@=11 \font\linear=line10 scaled \magstep5
\def\slant{{\linear ,}}
\def\young@lign{\everycr{}\tabskip0pt\halign}
\def\Mathstrut@{\setbox\z@\hbox{$($}\setbox\tw@\null\ht\tw@\ht\z@\dp\tw@\dp\z@
 \box\tw@}
\def\b@m#1{$\m@th\underline{#1}$}
\def\t@p#1{$\m@th\overline{#1}$}

\def\young{
\setbox\strutbox=\hbox{\vrule height3pt depth3.5pt width\z@}
 \offinterlineskip%
 {}\,\vcenter
\bgroup
     \def\1{{}}
     \def\2{\1&\1}
     \def\3{\1&\1&\1}
     \def\4{\1&\1&\1&\1}
\let\\=\cr
 \tabskip0pt\baselineskip0pt\m@th
 \young@lign
 \bgroup\vrule\b@m{\hbox to .75 em{\strut\hfil$##$\hfil}}\vrule %
   &&\b@m{\hbox to .75em{\strut\hfil$##$\hfil}}\vrule\crcr
   \noalign{\hrule}
  }

\def\endyoung{\egroup \egroup\,}
\def\slyoung{
   \setbox\strutbox=\hbox{\vrule height10pt depth2pt width\z@}
   \offinterlineskip%
   {}\,\vcenter
\bgroup
     \def\1{{}}
     \def\2{\1&\1}
     \def\3{\1&\1&\1}
     \def\4{\1&\1&\1&\1}
\let\\=\cr
 \tabskip0pt\baselineskip0pt\m@th
 \young@lign
 \bgroup\lower2pt\hbox{\slant}\kern-25pt\b@m{\hbox to 2.5em{\strut\hfil$##\;\;$}}\vrule
   &&\b@m{\hbox to 2.5em{\strut\hfil$##$\hfil}}\vrule\crcr
   \noalign{\hrule}
  }

\def\endslyoung{\egroup \egroup\,}
\def\frame #1#2#3#4{\vbox{\hrule height #1pt
 \hbox{\vrule width #1pt\kern #2pt
 \vbox{\kern #2pt
 \vbox{\hsize #3\noindent #4}
 \kern #2pt}
 \kern #2pt\vrule width #1pt}
 \hrule height0pt depth #1pt}}

\def\nframe #1#2#3#4{\vbox{
 \hrule height #1pt width0pt
 \hbox{\vrule height0pt width #1pt\kern #2pt
 \vbox{\kern #2pt
 \vbox{\hsize #3\noindent #4}
 \kern #2pt}
 \kern #2pt\vrule width #1pt height0pt}
 \hrule height0pt width0pt}}

\def\leftwall #1#2#3#4{\vbox{
 \hrule height #1pt width0pt
 \hbox{\vrule width #1pt\kern #2pt
 \vbox{\kern #2pt
 \vbox{\hsize #3\noindent #4}
 \kern #2pt}
 \kern #2pt\vrule width #1pt height0pt}
 \hrule height0pt width0pt}}

\def\rightwall #1#2#3#4{\vbox{
 \hrule height #1pt width0pt
 \hbox{\vrule height0pt width #1pt\kern #2pt
 \vbox{\kern #2pt
 \vbox{\hsize #3\noindent #4}
 \kern #2pt}
 \kern #2pt\vrule width #1pt }
 \hrule height0pt width0pt}}

\def\b #1{\frame{.3}{2}{8pt}{\centerline{#1}\vphantom{(}}}

\def\e{\vphantom{e}} 
\def\f{\b{\e}} 
\def\nl{\hfill\break} 

\begin{document}
\title{Matrix Pairs over Discrete Valuation Rings Determine
Littlewood-Richardson Fillings}
\author{Glenn D.\ Appleby (Corresponding Author), Tamsen Whitehead\\
Department of Mathematics\\
Santa Clara University\\
Santa Clara, CA 95053\\
{\tt gappleby@scu.edu, tmcginley@scu.edu}}
\date{ }
\maketitle
\newpage

\noindent {\bf Invariants of Matrix Pairs over Discrete Valuation Rings and
Littlewood-Richardson Fillings}

\medskip

\noindent
Glenn D. Appleby\\
 Tamsen Whitehead\\
{\em Department of Mathematics\\
and Computer Science,\\
Santa Clara University\\
Santa Clara,  CA 95053}\\
gappleby@scu.edu, tmcginley@scu.edu

\noindent \mbox{} \hrulefill \mbox{}
\begin{abstract}
Let $M$ and $N$ be two $r \times r$ matrices of full rank over a
discrete valuation ring $R$ with residue field of characteristic
zero.  Let $P,Q$ and $T$ be invertible $r \times r$ matrices over
$R$.  It is shown that the orbit of the pair $(M,N)$ under the
action $(M,N) \mapsto (PMQ^{-1}, QNT^{-1})$ possesses a discrete
invariant in the form of \LR\ fillings of the skew shape $\lambda /
\mu$ with content $\nu$, where $\mu$ is the partition of orders of
invariant factors of $M$, $\nu$ is the partition associated to $N$,
and $\lambda$ the partition of the product $MN$. That is, we may
interpret \LR\ fillings as a natural invariant of matrix pairs. This
result generalizes invariant factors of a single matrix under
equivalence, and is a converse of the construction in \cite{me}
where Littlewood-Richardson fillings were used to construct matrices
with prescribed invariants. We also construct an example, however,
of two matrix pairs that are not equivalent but still have the same
\LR\ filling.  The filling associated to an orbit is determined by
special quotients of determinants of a matrix in the orbit of the
pair.
\end{abstract}

\noindent \mbox{} \hrulefill \mbox{}

\section{Introduction and Example}
Let us briefly describe our results, and then provide complete definitions and
an example of our main construction.  It is well known that if $M$ and $N$ are
invertible matrices over a discrete valuation ring $R$, and if $\mu$ is the
partition of {\em orders} of the invariant factors of $M$ (with respect to a
fixed uniformizing parameter), with $\nu$ the partition for $N$, and $\lambda$
the partition for the product $MN$, then $c_{\mu \nu}^{\lambda}$, the \LR\
coefficient associated to this triple of partitions, is non-zero. This was
established in the module setting by T. Klein~\cite{klein} and investigated in
the matrix case by R.C. Thompson~\cite{Thomp-prod}. Given a triple of
partitions $(\mu, \nu, \lambda)$ such that $c_{\mu \nu}^{\lambda} \neq 0$,
Azenhas and Sa~\cite{AS} made an explicit construction of a matrix pair over a
discrete valuation $R$ whose orders of invariant factors correspond to the {\em
conjugate} of the partitions $\mu$, $\nu$, and $\lambda$. Later, the first
author~\cite{me} was able to produce, from a given \LR\ filling $\{ k_{ij} \}$,
a matrix pair $(M,N)$ such that the invariant factors of $M$ had orders $\mu$,
the invariant factors of $N$ had orders $\nu$, and those of the product $MN$
had orders $\lambda$.

In this paper we construct a converse to these results, of a sort.
Given a matrix pair $(M,N)$ of full rank over a discrete valuation
ring $R$, we will define a natural group action, generalizing matrix
equivalence for single matrices, and find a special pair $(D_{\mu},
N^*)$ in the orbit of $(M,N)$ from which orders of quotients of
determinants of $N^*$ will yield a \LR\ filling of the skew shape
$\lambda / \mu$ with content $\nu$ when the orders of the invariant
factors of $M$, $N$ and $MN$ are $\mu$, $\nu$, and $\lambda$.
Further, we show that this filling is an invariant (but not a
complete invariant) of the orbit of the pair $(M,N)$.

There has been an active interest in relating the combinatorics of
\LR\ fillings to other mathematical objects.  Survey papers by
Fulton \cite{fulton} and Zelevinsky~\cite{zel} demonstrate that
these combinatorial objects appear in a wide variety of contexts
including representation theory, the eigenvalue structure of
Hermitian matrices, and the Schubert calculus.   There is a fruitful
interplay between using the structures of a particular mathematical
context (in this case, the matrix algebra of discrete valuation
rings) to deepen our understanding of combinatorics, and also to use
combinatorial invariants to not only explain properties of interest
in matrix algebra, but to relate these algebraic questions to a
wider collection of problems.

The authors would like to express their appreciation for the careful
reading and very helpful suggestions made by the referee.
\bigskip

Let us now establish our notation and basic definitions.

 Let $R$ denote a discrete
valuation ring whose residue field is of
 characteristic zero.  There exists
an element $t \in R$ (called a {\em uniformisant} or {\em
uniformizing parameter}) with the property that every non-zero
element $a \in R$ can be written $a = ut^k$, where $k$ is a
non-negative integer and $u$ is unit in $R$ (let $R^{\times}$ denote
the units in $R$).  The choice of uniformizing parameter $t$ in $R$
is not unique, but given such a choice the decomposition $a=ut^k$,
when $u \in R^{\times}$, is uniquely determined.  (Basic facts
concerning discrete valuation rings may be found in~\cite{at-mac},
ch.\ 9 or~\cite{bour-ca}, ch.\ VI.)  Given $a \in R$, $a \neq 0$,
let us define the {\em order} of $a$, denoted $\| a \|$, to be:
\[ \|a\| = k, \quad \hbox{if} \quad a = ut^k, \quad \hbox{for} \ \ u
\in R^{\times}.
\] Note that if $\| x \| \neq \| y \|$, then
\[ \|x+y \| = \min \{ \|x \|, \|y \| \} \leq \| x \|, \| y \|. \]
If $\| x \| = \| y \|$, that is, if $x = u_{x}t^k$ and $y = u_y t^k$, the above
may fail. In particular, let $c_{*} : R \rightarrow R / tR$ denote the natural
map from $R$ to its residue field, identified with $R/(tR)$. Suppose $x = u_x
t^{k}$ and $y = u_y t^{k}$, for $u_x, u_y \in R^{\times}$. Then if
$c_{*}(u_{x}) = -c_{*}(u_y)$, we will have $\| x + y \| > \|x \| = \| y \|$. We
shall call this phenomenon ``catastrophic cancelation."  More generally, we
will say catastrophic cancelation has occurred in summing a collection of
elements $\{ x_{1}, \ldots , x_n \} \subseteq R$ whenever
\[ \min_{1 \leq i \leq n} \| x_{i} \| < \left\| \sum_{i=1}^{n}x_{i} \right\| . \]

Let $M_r (R)$ denote the ring of $r \times r$ matrices over $R$, and let $GL_r
(R)$ denote the invertible $r \times r$ matrices in $M_r (R)$ (that is,
matrices $M \in M_{r}(R)$ such that $\det (M) \in R^{\times}$). As is
well-known, for any $M \in M_{r}(R)$ of full rank there exist matrices $P, Q
\in GL_r (R)$ and an integer partition $\mu= (\mu_{1}, \ldots , \mu_{r})$ such
that
\[ PMQ^{-1} = \left[ \begin{array}{cccc}
t^{\mu_{1}} & 0 & \dots & 0 \\ 0 & t^{\mu_{2}} & \ddots &  \vdots \\
\vdots & \ddots & \ddots & 0 \\
0 & \dots & 0 & t^{\mu_{r}} \end{array} \right], \] where we may
assume $\mu_{1} \geq \mu_{2} \geq \dots \geq \mu_{r} \geq 0.$ (Note
we are writing the invariants in {\em decreasing} order.)  These
diagonal entries are the {\em invariant factors} associated to $M$
(see~\cite{hartley}, for example).
 The matrix $M$ uniquely determines the invariant factors, so we shall call
 the partition $(\mu_{1}, \mu_{2}, \ldots , \mu_{r})$, given by the {\em orders}
 of the invariant factors (with respect to $t$), the
{\em invariant partition} of $M$, and denote it by $inv\,(M) =
\mu=(\mu_{1}, \mu_{2}, \ldots , \mu_{r})$.

Here is an example of our main construction.  Precise definitions and
proofs will follow.   In this example, let $M$ already assume the diagonal form:
\[ M= \left[
\begin{array}{cccc} t^7 & 0 & 0 & 0 \\ 0 & t^4 & 0 & 0 \\ 0 & 0 & t^2 & 0\\ 0 & 0
& 0 & t \end{array} \right], \] so that $inv\,(M) =\mu=(7,4,2,1)$.  Then let
$N$ be the matrix
\[ N=
 \left[ \begin{array}{cccc}
t^{4} & t^4 & t^3 & t^2 \\
0 & t^{6} &  t^5+t^4 & t^4+2t^3\\
0 & 0& t^{5} & 2t^4+t^3 \\
 0 & 0  & 0 & t^{4}
\end{array} \right] .
\]
  The form of $N$ is not arbitrary.  It will be shown
  that a pair similar to $M$ and $N$ may be found in the orbit of {\em any} pair in a manner
  described below.  A standard calculation shows that $inv\,(N) = \nu =(8,5,4,2)$, and
that $inv\,(MN) = \lambda =  (11,10,7,5)$. Let us use the notation
\[ \| (i_1, \dots, i_k) \|, \]
to denote the order of the determinant of the submatrix of $N$ above, using
rows $i_1, \dots , i_k$, and the $k$ {\em right-most} columns.  So, for
example, $\| (1,2,4) \|$ will denote the order of the determinant of the
submatrix of $N$ with rows $1,2$, and $4$, using columns $2,3,$ and $4$.

Let us recursively define integers $k_{ij}$ by the following relations (we
define the order of the empty determinant $\|(\,)\|$ to be $0$):
\begin{eqnarray*}
k_{11} & = & \|(1, \,2 , \,3, \,4)\| -\|(2, \,3, \,4)\|= 4\\
k_{12} & = & \|(2, \,3 , \,4)\|-\|(1, \,3, \,4)\| =  2 \\
k_{13} & = & \|(1, \,3 , \,4)\|-\|(1, \,2, \,4)\|=1 \\
k_{14} & = & \|(1, \,2 , \,4)\|-\|(1, \,2, \,3)\|=1 \\
& & \\
k_{12}+k_{22} & = & \|(2, \,3 , \,4)\|-\|(3, \,4)\|=6 \\
k_{13}+k_{23} & = & \|(3, \,4 )\|-\|(1, \,4)\|=2 \\
k_{14}+k_{24} & = &\|(1, \,4 )\|-\|(1, \,2)\|=1 \\
& & \\
k_{13}+k_{23}+k_{33} & = & \|(3, \,4 )\|-\|(4)\|=5 \\
k_{14}+k_{24}+k_{34} & = & \|(4)\|-\|(1)\|=2 \\
& & \\
k_{14}+k_{24}+k_{34}+k_{44} & = & \|(4)\|-\|(\,)\|=4 .
\end{eqnarray*}

Note the telescoping of the sums in each group, so that, for instance, $k_{11}+
k_{12}+k_{13}+k_{14} = \|(1,2,3,4)\| - \|(1,2,3)\|$ and
$(k_{12}+k_{13}+k_{14})+(k_{22}+k_{23}+k_{24} )= \|(2,3,4)\| - \|(1,2)\|$.
Letting the $k_{ij}$ above denote the number of $i$'s in row $j$ in the skew
shape $\lambda / \mu,$ the matrix determinants above actually define a {\em \LR\
filling} of $\lambda / \mu$ with content $\nu$, as pictured in the diagram
below.  The boxes with no numbers in them form the partition $\mu$, which is
contained in the overall diagram of boxes $\lambda$:

\bigskip
\hspace{1.8in} \vbox{ \offinterlineskip \openup-1.5pt \nl
\f\f\f\f\f\f\f\b{$1$}\b{$1$}\b{$1$}\b{$1$} \nl
\f\f\f\f\b{$1$}\b{$1$}\b{$2$}\b{$2$}\b{$2$}\b{$2$}\nl
\f\f\b{$1$}\b{$2$}\b{$3$}\b{$3$}\b{$3$}\nl
\f\b{$1$}\b{$3$}\b{$4$}\b{$4$}\nl}

Further, this filling is uniquely determined by the matrix pair
$(M,N)$ up to a natural notion of equivalence, defined below.  We
will show that {\em any} given pair of matrices is equivalent to a
pair from which a system of determinantal formulas like the above
may be obtained to determine a particular \LR\ filling associated to
the orbit of the pair.

\section{Notation, Definitions, Background}

In what follows, given any partition $\alpha$, we shall let $\alpha_k$ denote
its $k$-th term, and assume that $\alpha_{k} \geq \alpha_{k+1}$. We shall also
write $\alpha \subseteq \beta$, for two partitions $\alpha$ and $\beta$, to
mean $\alpha_{k} \leq \beta_{k}$ for all $k \geq 1$.  This notation is
suggested by the fact that if we represent the partitions by non-increasing,
left-justified rows of boxes (called the {\em diagram} or {\em Young diagram}
of the partition), then $\alpha \subseteq \beta$ implies the diagram for
$\alpha$ fits inside the diagram of $\beta$. When $\alpha \subseteq \beta$, we
will denote by $\beta / \alpha$ the {\em skew diagram} consisting of the
diagram of $\beta$, with the diagram of $\alpha$ removed.  In the example
above, $\mu$ is depicted by the empty boxes, and the skew shape $\lambda / \mu$
consists of the boxes of $\lambda$ containing integers. Typically, partitions
are denoted by {\em infinite} decreasing sequences containing only finitely
many non-zero terms. The non-zero terms are the {\em parts} of the partition,
and the number of parts in a partition is its {\em length}.  We will denote the
length of partition $\mu$ by $length(\mu)$.
 The sum of the parts of a partition $\lambda$ is denoted $| \lambda |$ and is
called the {\em weight} of $\lambda$. See \cite{mac}, chapter 1, and
also~\cite{fulton-book},~\cite{sagan},~\cite{stanley}.

\medskip

We begin with the following combinatorial definition, which we shall relate to
matrices over $R$ presently.

\begin{df}
Let $\mu, \nu$, and $\lambda$ be partitions, with $length(\mu)$, $length(\nu) \leq
length(\lambda) \leq r$.  Let $S= (\lambda^{(0)},\lambda^{(1)}, \cdots ,
\lambda^{(r)})$ be a sequence of partitions in which $\lambda^{(0)} = \mu$.

The sequence $S$ is called a {\em \LR\ Sequence} of type $(\mu, \nu; \lambda)$
if there is a triangular array of integers $F=\{ k_{ij}: 1 \leq i \leq r, i\leq
j \leq r \}$ (called the {\em filling}) such that, for $1 \leq i \leq j, \ \
\lambda^{(i)}_{j} = \mu_{j} + k_{1j} + \dots + k_{ij}$, subject to the
conditions (LR1), (LR2), (LR3), and (LR4) below. We shall say, equivalently,
that any such set $F$ determines a {\em \LR\ filling} of the skew shape
$\lambda / \mu$ with content $\nu$. \label{LR-def}

(LR1) (Sums) For all $1 \leq i \leq j \leq r$,
\[  \mu_{j}+\sum_{s=1}^{j} k_{sj} = \lambda_{j}, \quad \hbox{and} \quad \sum_{s=i}^{r}k_{is} =
\nu_{i}. \]

(LR2)  (Non-negativity) For all $i$ and $j$, we have $k_{ij} \geq 0$.

(LR3) (Column Strictness) For each $j$, for $2 \leq j \leq r$ and $1 \leq i \leq j$ we require
$\lambda^{(i)}_{j} \leq \lambda^{(i-1)}_{(j-1)},$ that is,
\[  \mu_{j}+ k_{1j} + \cdots k_{ij} \leq \,  \mu_{(j-1)} +k_{1,(j-1)} +
 \cdots + k_{(i-1),(j-1)} . \label{cs}
\]

(LR4)  (Word Condition) For all $1 \leq i \leq r-1$, $i \leq j \leq r-1$,
\[   \sum_{s=i+1}^{j+1} k_{(i+1),s} \  \leq \  \sum_{s=i}^{j}k_{is}. \] \label{wd}
\end{df}
The first equality of (LR1) ensures that the sum of the number of boxes in row
$j$ of the filled diagram (including the empty boxes of the parts of $\mu$) sum
to $\lambda_{j}$, the $j$-th part of the partition $\lambda$, while in the
second equality we require that the sum of the number of $i$'s in all the rows
is $\nu_{i}$, the $i$-th part of the partition $\nu$.  The condition (LR2) is
included here because the non-negativity of the $k_{ij}$ will not be obvious
from the definition we shall adopt, and will need to be proved. (LR3) says that
the numbers in the filling are strictly increasing down columns. Lastly, (LR4)
indicates that the number of $i$'s in rows $i$ through $j$ is greater than or
equal to the number of $(i+1)$'s in rows $(i+1)$ through $(j+1)$.

\begin{df} Given partitions $\mu$, $\nu$, and $\lambda$, we shall
let $c_{\mu \nu}^{\lambda}$ denote the number of \LR\ fillings of the skew shape $\lambda / \mu$ with content $\nu$.  The non-negative integer $c_{\mu \nu}^{ \lambda}$ is
called the {\em \LR\ coefficient} of the partitions $\mu$, $\nu$, and
$\lambda$.
\end{df}

We begin with the well-known result relating Littlewood-Richardson coefficients
to invariants of matrices over discrete valuation rings.

\begin{prop}[\cite{me},\cite{AS},\cite{klein},\cite{mac},\cite{Thomp-prod}]
Let $R$ be a discrete valuation ring and let $\mu, \nu,$ and $\lambda$ be
partitions of length $r$. If $c_{\mu \nu}^{\lambda}
> 0$, then we may find $r \times r$ matrices $M$ and $N$ over $R$
such that $inv\,(M) = \mu, inv\,(N) = \nu$, and $inv\,(MN) = \lambda$, and
conversely. \label{mat-lr}
\end{prop}
\bigskip

\noindent
\section{Matrix Realizations of \LR\ Fillings}

By Proposition~\ref{mat-lr} above, if $F=\{ k_{ij}: 1 \leq i \leq r,
\, i \leq j \leq r \}$ is a \LR\ filling of $\lambda / \mu$ with
content $\nu$, then there must exist matrices $M$ and $N$, over $R$
so that $inv\,(M) = \mu$, $inv\,(N) = \nu$, and $inv\,(MN) =
\lambda$. We would like to see how a specific \LR\ filling
determines such matrices.

\begin{df}
Suppose we are given a \LR\ filling $F$, with associated \LR\ sequence $S$. By
a {\em factored matrix realization} for the \LR\ filling $F$ of $\lambda / \mu$
with content $\nu$ (or the \LR\ sequence $(\lambda^{(0)},\lambda^{(1)} \ldots , \lambda^{(r)})$ of type $(\mu,\nu;
\lambda)$ ) we shall mean a set of $r+1$ matrices $M,N_{1},N_{2}, \ldots ,
N_{r}$ so that \label{LR-def}
\begin{enumerate}
\item $inv\,(M) = \mu= \lambda^{(0)}$
\item $inv\,(N_{1}N_{2} \cdots N_{i}) =
(\nu_{1}, \nu_{2}, \ldots, \nu_{i}, 0,0,\ldots )$ for all $i \leq r$.  So, in
particular, $inv\,(N_1\cdots N_r) = \nu$.
\item $inv\,(MN_{1}N_{2} \cdots N_{i}) = \lambda^{(i)}$, for $1 \leq i \leq r$.
\end{enumerate}
\end{df}

Factored matrix realizations code up the \LR\ filling of the skew diagram
$\lambda / \mu$. In \cite{me} a simple construction of a factored matrix
realization was obtained from a given \LR\ filling.  This result, though based
on conjugate sequences, had first been obtained in \cite{AS}.

\begin{thm}[\cite{me}]
Let $F=\{ k_{ij}: 1 \leq i  \leq j \leq r \}$ be a \LR\ filling of $\lambda /
\mu$ with content $\nu$. Define $r \times r$ matrices $M,N_{1},N_{2}, \ldots
N_{r}$ over $R$ by \label{my-thm}
\begin{enumerate}
\item $M = diag(t^{\mu_{1}},t^{\mu_{2}}, \ldots , t^{\mu_{r}})$,
where $ \mu_{1} \geq \mu_{2} \geq \cdots \geq \mu_{r}\geq 0$.
\item Define the block matrix $N_{i}$ by
\[ N_{i} = \left[ \begin{array}{c|c}
1_{i-1} & 0 \\
\hline 0 & T_{i}
\end{array} \right] \]
where $T_{i}$ is the $(r-i+1) \times (r-i+1)$ matrix:
\[ T_{i} = \left[ \begin{array}{ccccc}
t^{k_{i,i}}    & 1    & 0    & \cdots    &       0  \\
0        & t^{k_{i,i+1}} & 1 &  \ddots   &       \vdots     \\
0 &      0& t^{k_{i,i+2}} & \ddots  &        0  \\
 \vdots & \vdots  & \ddots  &  \ddots    & 1 \\
0  &  0  & \cdots  &  0 &    t^{k_{i,r}} \\

\end{array} \right] \]
and $1_{i-1}$ is an $(i-1) \times (i-1)$ identity matrix.
\end{enumerate}
Then $M, N_{1}, N_{2}, \ldots , N_{r}$ is a factored matrix realization of the
Little\-wood\--Rich\-ard\-son filling $F$. \label{lr_to_mat}
\end{thm}

(Note that in \cite{me}, Theorem~\ref{my-thm} was written so that invariants
were calculated in {\em increasing} order, and so the matrices used in the
factorization have a slightly different form.)

Let's consider our main example.  Recall $\mu = (7,4,2,1)$, $inv\,(N) = \nu
=(8,5,4,2)$, and that $inv\,(MN) = \lambda =  (11,10,7,5)$. We will use the
following \LR\ filling of the skew diagram $ \lambda / \mu$:

\bigskip
\hspace{1.4in} \vbox{ \offinterlineskip \openup-1.5pt \nl
\f\f\f\f\f\f\f\b{$1$}\b{$1$}\b{$1$}\b{$1$} \nl
\f\f\f\f\b{$1$}\b{$1$}\b{$2$}\b{$2$}\b{$2$}\b{$2$}\nl
\f\f\b{$1$}\b{$2$}\b{$3$}\b{$3$}\b{$3$}\nl \f\b{$1$}\b{$3$}\b{$4$}\b{$4$}\nl}
\bigskip

So, for instance, $k_{11} = 4$, $k_{12} = 2$, $k_{13} = 1$, and $k_{14}= 1$.
Then by Theorem~\ref{lr_to_mat} we define
\[ M = \left[ \begin{array}{cccc}
t^7 & 0 & 0 & 0 \\
0 & t^{4} & 0 & 0 \\
0 & 0 & t^{2} & 0 \\
0 & 0 & 0 & t
\end{array} \right] \]
and
\[ N_{1} = \left[ \begin{array}{cccc}
t^4 & 1 & 0 & 0\\
0 & t^{2} & 1 & 0 \\
0 & 0 & t^{1} & 1 \\
0 & 0 & 0 & t
\end{array} \right], \quad N_{2} = \left[
\begin{array}{c|ccc}
1 & 0 & 0 & 0  \\ \hline
0& t^4 & 1 & 0  \\
0&  0& t^{1} & 1  \\
0& 0 &0 &  t^0  \\
\end{array} \right],
N_{3} = \left[ \begin{array}{cc|cc}
1 & 0 & 0 & 0 \\
0 & 1 &  0 &0 \\
\hline 0 & 0 & t^3 & 1 \\
0 & 0 & 0 & t
\end{array} \right], \quad N_{4} = \left[
\begin{array}{ccc|c}
1 & 0 & 0 & 0 \\
0 &  1 & 0 & 0 \\
0 &  0  &   1 & 0 \\
\hline 0 & 0 & 0 & t^2
\end{array} \right]. \]
Then define $N$ by
\[ N  = N_{1}N_{2}N_{3}N_{4} = \left[ \begin{array}{cccc}
t^{4} & t^4 & t^3 & t^2 \\
0 & t^{6} &  t^5+t^4 & t^4+2t^3\\
0 & 0& t^{5} & 2t^4+t^3 \\
 0 & 0  & 0 & t^{4}
\end{array} \right] , \]
and we recover the matrix of our main example.  Clearly $inv\,(M) = \mu$. The
orders of entries in the product matrix $N$ increase as one proceeds to the
left in any row, and down any column. However, we find that in the product
\[ MN =\left[ \begin{array}{cccc}
t^{11} & t^{11} & t^{10} & t^{9} \\
0 & t^{10} &  t^9+t^8 & t^8+2t^7\\
0 & 0& t^{7} & 2t^6+t^5 \\
 0 & 0  & 0 & t^{5}
\end{array} \right], \]
the orders now increase as we proceed {\em up} any column.  From this it is
easily seen that $inv\,(MN) = (11,10,7,5) = \lambda$ (the orders of the
diagonal entries). An easy calculation shows that $inv\,(N) = (8,5,4,2)= \nu$.
Note also that $\|(N)_{i,i}\| = \lambda_{i} - \mu_{i}, $ and the orders of the
entries along the top row satisfy $ \|(N)_{1,i}\| = k_{i, i} $.

\medskip

So, given a Littlewood-Richardson filling of the skew-shape $\lambda / \mu$
with content $\nu$, we are able to construct a pair of matrices $M, N \in
M_{n}(R)$ such that $inv\,(M) = \mu, inv\,(N) = \nu$, and $inv\,(MN) =
\lambda$.

\section{The $\mu$-Generic Form for Matrix
Pairs} In this paper we shall prove a converse to Theorem~\ref{my-thm}.  Let
${\cal M}_{r}^{(2)}$ denote the set of all pairs $(M,N)$ of $r \times r$
matrices over $R$ of {\em full rank}. We shall show that every matrix pair
$(M,N) \in {\cal M}_{r}^{(2)}$, such that $inv\,(M) = \mu, inv\,(N) = \nu$ and
$inv\,(MN) = \lambda$, determines a Littlewood-Richardson filling of the
skew shape $\lambda / \mu$ with content $\nu$.  In fact, we can say more.

Given a partition $\mu = ( \mu_{1} , \mu_{2}, \ldots , \mu_{r} )$, define
\[ D_{\mu} = diag(t^{\mu_{1}}, \dots , t^{\mu_{r}} ) = \left[
\begin{array}{cccc} t^{\mu_{1}} & 0 & \cdots &
 0 \\
 0 & t^{\mu_{2}} & \ddots & \vdots \\
 \vdots & \ddots &\ddots & 0 \\
 0 & \cdots & 0  & t^{\mu_{r}} \end{array} \right] . \]

\begin{df}

\begin{enumerate}
\item Let $(P,Q,T) \in GL_{r}(R)^3$ be a triple of invertible
matrices (where $GL_{r}(R)^3$ forms a group under multiplication in each
component).
 Let $(M,N) \in {\cal M}_{r}^{(2)}$, and then define the map
 $GL_{r}(R)^3 \times {\cal M}_{r}^{(2)} \rightarrow {\cal
 M}_{r}^{(2)}$ by
 \[ (P,Q,T) \cdot (M,N) = (PMQ^{-1}, QNT^{-1}). \]
 \item If $(M',N') =(P,Q,T) \cdot (M,N)$ for some triple of
 invertible matrices $(P,Q,T) \in GL_{r}(R)^3$, we will say the pair
 $(M',
 N')$ is {\em pair equivalent} to $(M,N)$.
\end{enumerate}
\end{df}

From these definitions, the following results are easily checked.

\begin{prop}
\label{orbit}
\begin{enumerate}
\item The mapping
\[ (P,Q,T) \cdot (M,N) = (PMQ^{-1}, QNT^{-1}) \]
is a group action of $GL_{r}(R)^3$ on the set of pairs of full-rank matrices
${\cal M}_{r}^{(2)}$.
\item ``Pair equivalence'' is an equivalence relation on
${\cal M}_{r}^{(2)}$.
\item Every pair $(M,N) \in {\cal M}_{r}^{(2)}$ is pair equivalent to a
pair $(D_{\mu}, N')$, where $D_{\mu} = diag(t^{\mu_{1}}, \ldots ,
t^{\mu_{r}})$, and $\mu = (\mu_{1}, \ldots , \mu_{r}) = inv\,(M)$.
\end{enumerate}
\end{prop}

This action of $GL_{r}(R)^3$ on ${\cal M}_{r}^{(2)}$ is chosen in order to
preserve the invariant partitions of  $M$ and $N$, and also $MN$, so that if
$(M,N)\in {\cal M}_{r}^{(2)}$ is pair equivalent to $(M',N')$, then $\mu =
inv\,(M) = inv\,(M')$, $\nu = inv\,(N) = inv\,(N')$, and $\lambda = inv\,(MN) =
inv\,(M'N')$.

In this paper we shall show that the pair equivalence class of
$(M,N) \in {\cal M}_{r}^{(2)}$, when $\mu = inv\,(M)$, $\nu =
inv(N)$, and $\lambda = inv\,(MN)$, uniquely determines a
Littlewood-Richardson filling of the skew-shape $\lambda / \mu$ with
content $\nu$. The construction we present here can also be applied
to the pair to obtain a \LR\ filling of $\lambda / \nu$ with content
$\mu$, providing yet another proof that $c_{\mu \nu}^{\lambda} \neq
0$ implies $c_{\nu \mu}^{\lambda} \neq 0$.  It appears (\cite{new})
that the pair of fillings associated to a matrix pair are in {\em
bijection}, so that a given filling of $\lambda / \mu$ with content
$\nu$ occurs with a {\em uniquely} determined filling of $\lambda /
\nu$ with content $\mu$, independent of the particular matrix
realization of it. In fact, the matrix setting recovers the same
combinatorially determined bijections found in \cite{Sottile}, and
also~\cite{pak-2}.

A given pair of fillings is not a complete invariant of the orbit of a matrix
pair, however. In particular, we will in show in Section 6 that there are
distinct pairs $(M,N)$ and $(M',N')$ which give rise to the {\em same}
Littlewood-Richardson filling, but which are {\em not} pair equivalent. This
suggests the set of orbits possesses a more intricate structure, for which the
\LR\ fillings provide a discrete invariant. We shall not pursue this further
here, though it does appear that the orbits might be classified by a collection
of continuously varying parameters within a collection of orbits for which the
discrete invariants have been fixed.

To continue, let
\[ GL_{r}(R)^3 \backslash {\cal M}_{r}^{(2)} \]
denote the set of pair equivalence classes.  It is clear we may decompose this
set as a disjoint union of $GL_{r}(R)^3$-invariant orbits according to the
triple of invariant partitions $(\mu,\nu,\lambda)$ associated to a matrix pair:
\[GL_{r}(R)^3 \backslash {\cal M}_{r}^{(2)} \leftrightarrow
\coprod_{\mu,\nu,\lambda} GL_{r}(R)^3 \backslash ({\cal
M}_{r}^{(2)})_{\mu,\nu,\lambda}, \] where
\[({\cal M}_{r}^{(2)})_{\mu,\nu,\lambda} = \Big\{ (M,N)\in {\cal M}_{r}^{(2)} :\hbox{$\mu =
inv\,(M)$, $\nu = inv\,(N)$, and $\lambda = inv\,(MN)$} \Big\}. \]

Fix, for now, a triple of partitions $(\mu, \nu, \lambda)$ which determines a
set of orbits denoted by $GL_{r}(R)^3 \backslash ({\cal
M}_{r}^{(2)})_{\mu,\nu,\lambda}$.  It is clear from Proposition~\ref{orbit}
that each orbit contains a pair $(D_{\mu},N)$ such that $inv\,(N) = \nu$ and
$inv\,(D_{\mu}N) = \lambda$.  Let ${\cal G}_{\mu}$ denote the subgroup
stabilizing the first term $D_{\mu}$:
\[ {\cal G}_{\mu} = \Big\{ (P,Q,T) \in GL_{r}(R)^3 : \hbox{for all
$N$,} \, (P,Q,T)\cdot (D_{\mu},N) = (D_{\mu},N'),\ \hbox{for some} \
N' \in M_{n}(R) \Big\}.
\] There is a natural bijection
\[GL_{r}(R)^3 \backslash ({\cal
M}_{r}^{(2)})_{\mu,\nu,\lambda} \leftrightarrow {\cal G}_{\mu} \backslash
\Big\{ (D_{\mu}, N) : inv\,(N) = \nu, inv\,(D_{\mu}N) = \lambda \Big\}. \]

However, $(P,Q,T) \in {\cal G}_{\mu}$ if and only if
\begin{align*}
(P,Q,T) \cdot (D_{\mu},N) & = (PD_{\mu}Q^{-1}, QNT^{-1}) \\
& = (D_{\mu},QNT^{-1}), \end{align*} so that
\[ D_{\mu} = PD_{\mu}Q^{-1}, \]
and so
\[ P = D_{\mu}QD_{\mu}^{-1} \in GL_{r}(R). \]
\begin{df}
If $Q \in GL_{r}(R)$ satisfies
\[ D_{\mu}Q D_{\mu}^{-1} \in GL_{r}(R), \]
we shall say $Q$ is {\em $\mu$-admissible.}
\end{df}

Note first that the set of $\mu$-admissible matrices forms a group under
multiplication, so we shall denote this group by:
\[ G_{\mu} = \Big\{ Q \in GL_{r}(R) : \hbox{$Q$ is $\mu$-admissible} \Big\}. \]

 Note also that if $Q$ is
$\mu$-admissible, then $D_{\mu}QD_{\mu}^{-1} = P \in GL_{r}(R)$, so that
\[ Q = D_{\mu}^{-1}P D_{\mu}. \] Thus, if $q_{ij}$ denotes the $(i,j)$ entry of
$Q$, we must have
\[ \| q_{ij} \| \geq \mu_{j} - \mu_{i}, \ \ i > j. \]  In particular, if $Q = Q_{L}$ is
itself a lower triangular, $\mu$-admissible matrix, we may write
\[ Q_{L} = D_{\mu}^{-1} Q_{L}^{0}D_{\mu}, \]
for some invertible, lower triangular $Q_{L}^{0} \in GL_{r}(R)$.

Since $(P,Q,T) \in {\cal G}_{\mu}$ if and only if $P = D_{\mu}QD_{\mu}^{-1}$
for some $\mu$-admissible $Q$, it is sufficient, when seeking invariants of
orbits $GL_{r}(R)^3 \backslash ({\cal M}_{r}^{(2)})_{\mu,\nu,\lambda}$ to
consider
 the
natural bijection with the set of orbits:
\[ \left( G_{\mu} \times GL_{r}(R) \right) \backslash {\cal
M}_{\mu,\nu,\lambda}, \] where
\[ \quad {\cal M}_{\mu,\nu,\lambda} = \Big\{ N \in M_{r}(R) : inv\,(N) =
\nu, \ inv\,(D_{\mu}N) = \lambda \Big\}, \] and where we define the action of
$G_{\mu} \times GL_{r}(R)$ on ${\cal M}_{\mu,\nu,\lambda}$ by
\[ (Q,T)\cdot N = QNT^{-1}. \]

The substance of our results in this paper will be to find, in the orbit of $N$ under the group $G_{\mu} \times GL_{r}(R)$, a matrix in a special form that we will call ``$\mu$-generic".  As in our main example, we will use differences of orders of determinants of a $\mu$-generic matrix in the orbit of $N$ to define a \LR\ filling uniquely associated to this orbit.

Note:  In order to relate invariant factors to the partitions used
in Littlewood-Richardson tableaux, we restricted our attention to
full-rank matrices ${\cal M}_{r}^{(2)}$.  Some preliminary
investigations suggest that many of the matrix-theoretic results
presented here generalize to matrices over $R$ of arbitrary rank.
Extending the combinatorial interpretation to this case would
necessitate, it seems, considering diagrams with rows of
``infinite'' length (by regarding a $0$ among the invariant factors
as $0 = t^{\infty}$). Such a view may be possible and interesting,
but is not taken up here.

\medskip

In this Section we shall prove that an arbitrary matrix pair $(M,N)
\in {\cal M}_{r}^{(2)}$ is pair-equivalent to a pair $(D_{\mu},N^*)$
from which, as we will show in Section 5, a \LR\ filling may be
determined. We will begin by recording some definitions and
preliminary lemmas that will be used throughout the paper.

\begin{df}
 Let $I$, $J$, and $H$ be subsets of $ \{ 1,2,\ldots, r \}$ of length
$k$, written as $I = (i_{1}, i_{2}, \ldots , i_{k})$,  where $1 \leq i_{1} <
i_{2} < \cdots < i_{k} \leq r$, and similarly for $J$ and $H$. We call such
sets {\em index sets}.  (Note: $I,J$, and $H$ do {\em not} denote partitions.)
Let $I \subseteq H$ denote the condition that $i_{s} \leq h_{s}$ for $1 \leq s
\leq k$. Given an $r \times r$ matrix $W$, let \label{det-def}
\[ W_{IJ}=W\left( \begin{array}{cccc} i_{i} & i_{2} & \cdots & i_{k} \\
j_{1} & j_{2} & \cdots & j_{k} \end{array} \right) \] denote the $k \times k$
minor of $W$ using rows $I$ and columns $J$ (that is, the determinant of this
submatrix).  \label{mat-def} Let us extend the definition of $\|a\|$ to square
matrices, so that if $B$ is any square matrix, $\| B\|$ will denote $\| \det(B)
\|$.

Also, given a partition $\mu = (\mu_{1}, \ldots , \mu_{r})$, let $\mu_{I}$
denote the partition $\mu_{I}=(\mu_{i_{1}}, \mu_{i_{2}}, \ldots, \mu_{i_{k}})$,
and let $|\mu_{I}|= \mu_{i_{1}} + \mu_{i_{2}} + \cdots + \mu_{i_{k}}$.
\end{df}



We will also need the following result.  It provides one of the demonstrably
{\em least} efficient methods to compute an $LU$ decomposition of a matrix over
a field.  We shall apply it in our case to matrices over a discrete valuation
ring, so we should interpret the lemma below in terms of factorizations over
the field of fractions of $R$ (even though, in the cases that will be important
to us we will show that the factors are actually defined over the ring $R$).

\begin{lem}[\cite{Gant}, pp.\ 35-36.]
Every $r \times r$ matrix $A=(a_{ij})$ of rank $s$ in which the first $s$
successive principal minors are different from zero:
\[ D_{k} =  A \left( \begin{array}{c} 1 \ 2 \ \dots \ k \\
1 \ 2 \ \dots \ k \end{array} \right)  \neq 0, \quad \hbox{for $k=1,2, \ldots,
s$,} \] can be represented as a product of a lower-triangular matrix $B$ and an
upper-triangular matrix $C$:
\[ A = BC = \left[ \begin{array}{cccc}
b_{11} & 0 & \cdots & 0 \\
b_{21} & b_{22} & \ddots  & \vdots\\
\vdots & \vdots & \ddots & 0 \\
b_{r1} & b_{r2}& \cdots & b_{rr} \end{array} \right] \left[
\begin{array}{cccc}
c_{11} & c_{12} & \cdots & c_{1r} \\
0 & c_{22} & \cdots  & c_{2r} \\
\vdots & \ddots & \ddots & \vdots \\
0 & \dots & 0 & c_{rr} \end{array} \right] . \] Here
\[ b_{11}c_{11} = D_{1}, \ b_{22}c_{22} = \frac{D_{2}}{D_{1}}, \cdots
, b_{ss}c_{ss} = \frac{D_{s}}{D_{s-1}}. \] The values of the diagonal elements
of $B$ and $C$ can be chosen arbitrarily subject to the conditions above.

When the diagonal elements of $B$ and $C$ are given, then the elements in $B$
and $C$ are uniquely determined, and are given by the following formulas:
\[ b_{gk}=b_{kk}\frac{ A \left( \begin{array}{c} 1 \ 2 \
\dots \ k-1 \ g \\
1 \ 2 \ \dots \ k-1 \ k \end{array} \right)}{A \left(
\begin{array}{c} 1 \ 2 \ \dots \ k \\
1 \ 2 \ \dots \ k \end{array} \right) } , \quad c_{kg}= c_{kk}\frac{ A
\left( \begin{array}{c} 1 \ 2 \ \dots \ k-1 \ k \\
1 \ 2 \ \dots \ k-1 \ g \end{array} \right)  }{  A \left(
\begin{array}{c} 1 \ 2 \ \dots \ k \\
1 \ 2 \ \dots \ k \end{array} \right) } , \] for $k=1, 2, \ldots , s$, and
$g=k, k+1, \ldots , r$ .\label{gant}
\end{lem}

As we have seen, any matrix pair $(M,N) \in M_{r}^{(2)}$ is pair equivalent to
$(D_{\mu}, N')$, so it is sufficient to work with the action of the group
$G_{\mu} \times GL_{r}(R)$ (where $G_{\mu}$ denotes the group of
$\mu$-admissible matrices) acting on matrices $N$ via $(Q,T)\cdot N = QNT^{-1}$
for $(Q,t) \in G_{\mu} \times GL_{r}(R)$.  We shall prove the existence of a
matrix $N^*$ in the orbit of this action that is, in a sense to be made precise
below, ``generic'', and from which determinantal formulas similar to those in
our main example will allow us to compute a \LR\ filling associated to this
orbit, and hence to the pair $(M,N)$.  The following lemma will prove the
existence of this generic matrix, and the proposition to follow will establish
the key determinantal inequalities on which our method depends.

\begin{lem}
Let $N \in M_{r}(R)$ be full-rank.  Then there are matrices
$Q_{U},Q_{L},T_{L}$ and $T_{U}$ such that $Q_{U}$ and $Q_{L}$ are
$\mu$-admissible, $Q_{U}$ and $T_{U}$ are upper triangular, and
$Q_{L}$ is lower triangular and $T_{L}$ is a permutation matrix that
is multiplied on the right by a lower triangular matrix. We will set
\[ Q = Q_{U}Q_{L} \qquad \hbox{and} \qquad T^{-1} = T_{L}T_{U}. \]
(Note that $Q$ is written deliberately, and atypically in a ``UL''
decomposition.)  Then we may choose $Q_{U},Q_{L}T_{L}$ and $T_{U}$ subject to
the conditions above so that:
\begin{enumerate}
\item The ``$LU$'' factorization $Q =
\widehat{Q_{L}}\widehat{Q_{U}}$ exists and is defined over $R$, for
$\widehat{Q_{L}}$ and $\widehat{Q_{U}}$ lower and upper triangular,
$\mu$-admissible matrices, respectively,
\item $Q_{L}NT_{L}$, and hence $Q_{U}Q_{L}NT_{L}T_{U}$, is upper triangular,
\item For any index sets $I,J$ of length $k \leq r$:
\label{det-lem}
\begin{align}
\| (QNT^{-1})_{IJ} \| = \| (Q_{U}Q_{L}NT_{L}T_{U})_{IJ}\|
& = \min_{I \subseteq S} \|(Q_{L} N T^{-1})_{SJ} \| \label{first}\\
&  = \min_{H \subseteq I} \Big\{ \| (\widehat{Q_{U}}NT^{-1})_{HJ} \|
+ | \mu_{H} | - | \mu_{I} | \Big\}  \label{second} \\
& = \min_{H \subseteq J} \left\| (QNT_{L})_{IH} \right\|.
\label{third}
\end{align}
\end{enumerate}
\end{lem}
\begin{proof}Note that since $Q_{L}$ is required to be a lower triangular,
$\mu$-admissible matrix, there is a uniquely determined lower triangular
invertible matrix $Q_{L}^{0}$ such that $D_{\mu}^{-1}Q_{L}^{0}D_{\mu} = Q_{L}$.

For any matrix $W \in M_{r}(R)$, let us denote by $c_{*}(W)$ the
matrix taking values in the residue field of $R$ obtained by
applying $c_{*}$ to each entry of $W$.

Our method will be to determine, for all index sets $I$ and $J$, a collection
of finitely many polynomials in the entries of $c_{*}(Q_{L}^{0}),
c_{*}(Q_{U})$, $c_{*}(T_{L})$ and $c_{*}(T_{U})$ with coefficients determined
by the $c_{*}$ images of minors of $N$, such that if the $c_{*}$ images of the
entries of these matrices lie {\em outside} the variety defined by the common
solutions of these polynomials, Equations~\ref{first},~\ref{second}
and~\ref{third} will be satisfied. Since these polynomials will generate a
proper ideal and our residue field is infinite, the existence of a solution
will be assured.

We begin by establishing some facts obtainable for any choice of
matrices $Q_{U},Q_{L},T_{L}$ and $T_{U}$.  We will always denote $Q$
by the matrix $Q=Q_{U}Q_{L}$ and $T^{-1}$ by $T^{-1}=T_{L}T_{U}$.
Then, by the Cauchy-Binet formula we have:

\begin{align}
(QNT^{-1})_{IJ}  = (Q_{U}Q_{L}NT_{L}T_{U})_{IJ}    & = \sum_{I \subseteq S, H
\subseteq J} (Q_{U})_{IS} (Q_{L}NT_{L})_{SH}(T_{U})_{HJ}. \label{secondvar}
\end{align}

 Note that the conditions $I \subseteq S$
and $H \subseteq J$ are necessary since we require $Q_{U}$ and $T_{U}$ to be
upper triangular, and so these minors would vanish were these conditions not
satisfied.  We will first require that the entries in $Q_{U}$ and $T_{U}$ are
all units over $R$.  In fact, we may require that $(Q_{U})_{IS}, (T_{U})_{HJ}
\in R^{\times}$ for all $I,J,S,H$ since this amounts to requiring that the
$c_{*}$ images of $Q_{U}$ and $T_{U}$ lie outside the varieties $\det(
(c_{*}(Q_{U}))_{IS}=0$ and $\det((c_{*}(T_{U})_{HJ}=0$ for all index sets
$I,J,S,H$.  (In all, we shall make many successive ``requirements'' on the
matrices we discuss.  The point will be that all these requirements are open
polynomial conditions and so may be simultaneously met.)  The upshot of
Equation~\ref{secondvar} is now that $(QNT^{-1})_{IJ}$ may be written as a sum
of {\em unit} multiples of terms $(Q_{L}NT_{L})_{SH}$. We claim that we may
choose the units $(Q_{U})_{IS}$ and $(T_{U})_{HJ}$ to avoid any catastrophic
cancelation in the sum appearing in Equation~\ref{secondvar}.  In fact, we can
ensure that no catastrophic cancelation occurs in {\em any subset} of terms in
this sum. Let ${\cal S}$ be a subset of index sets of length $k$ such that $I
\subseteq S$, and similarly let ${\cal H}$ be a collection of index sets only
constrained by the condition $H \in {\cal H}$ implies $H \subseteq J$.  Let
$m_{0}$ be defined as
\[ m_{0} = \min_{S \in {\cal S}, H \in {\cal H}} \| (Q_{L}NT_{L})_{SH}
\| . \] Then, let us define the function $c_{*, m_{0}}: R
\rightarrow R/ (tR)$ by setting
\[ c_{*,m_{0}}(a) = \begin{cases} c_{*}(a/t^{m_{0}}) & \text{if $\|a
\| = m_{0}$} \\
0 & \text{otherwise} \end{cases}. \]

The existence of catastrophic cancelation in Equation~\ref{secondvar} may now
be expressed by the condition:
\begin{equation} \sum_{(S,H) \in {\cal M}}
c_{*}((Q_{U})_{IS})\, c_{*,m_{0}}((Q_{L}NT_{L})_{SH})\, c_{*}((T_{U})_{HJ}) =
0. \label{cat} \end{equation}

Given any $N$, along with fixed choice of $Q_{L}$ and $T_{L}$, we may certainly
choose matrices $Q_{U}$ and $T_{U}$ so that the units $c_{*}((Q_{U})_{IS})$ and
$c_{*}((T_{U})_{HJ})$ lie outside the variety defined by Equation~\ref{cat}
above.  We shall, in fact, require our $Q_{U}$ and $T_{U}$ to lie outside all
the varieties defined by the (finitely many) choices of index sets $I$ and $J$
and sets of index sets ${\cal S}$ and ${\cal H}$, so that there will be no
catastrophic cancelation among collections of terms appearing in any equation
of type of Equation~\ref{cat} above, once we choose $Q_{L}$ and $T_{L}$.  We
shall describe this as a {\em generic} choice of $Q_{U}$ and $T_{U}$, with
respect to some choice of $N$, $Q_{L}$ and $T_{L}$.

Let us choose a $\mu$-admissible $Q_{L}$ so that if we write $Q_{L} =
D_{\mu}^{-1}Q_{L}^{0}D_{\mu}$, we may assume all $\| (Q_{L}^{0})_{IH} \| = 0$,
which as before is a polynomially open condition.  Given $Q_{L}$,
let us define $T_{L}$ by the requirement that $Q_{L}NT_{L}$ is upper
triangular.  This, after a possible permutation of columns, is
obtainable by a lower triangular transformation.  With these fixed,
let us choose $Q_{U}$ and $T_{U}$ to be generic in the sense
described above.  Then, we may write

\begin{align}
(QNT^{-1})_{IJ}  = (Q_{U}Q_{L}NT_{L}T_{U})_{IJ}    & = \sum_{I
\subseteq S, H \subseteq J} (Q_{U})_{IS}
(Q_{L}NT_{L})_{SH}(T_{U})_{HJ}  \nonumber \\
& =\sum_{I \subseteq S} (Q_{U})_{IS}\Big( \sum_{H \subseteq
J}(Q_{L}NT_{L})_{SH}(T_{U})_{HJ} \Big) \nonumber  \\
& =\sum_{I \subseteq S} (Q_{U})_{IS} (Q_{L}NT^{-1})_{SJ}.
\end{align}

Since there can be no catastrophic cancelation among the terms
appearing above, the order of $(QNT^{-1})_{IJ}$ must be the minimum
of the orders of the $(Q_{L}NT^{-1})_{SJ}$, so Equation~\ref{first}
is satisfied.

We shall continue with Equation~\ref{second}.  Let us first note
that in order for the ``LU'' factorization $
Q=\widehat{Q_{L}}\widehat{Q_{U}}$ to exist for the matrix
$Q=Q_{U}Q_{L}$, it is sufficient that the principal minors of $Q$ be
units in $R$.  It is easy to show this is a polynomially open
condition on $c_{*}(Q_{U})$ and $c_{*}(Q_{L})$ (and noting that all
entries on and above the principal submatrices in the product $Q$
will be units), so we may ensure this factorization exists and is
defined over $R$.

Since $Q_{U}$ and $Q_{L}$ are $\mu$-admissible, so is the product
$Q=Q_{U}Q_{L}$, hence so is the product
$\widehat{Q_{L}}\widehat{Q_{U}}$.  Since every invertible upper
triangular matrix, such as $\widehat{Q_{U}}$, is automatically
$\mu$-admissible, it follows that $\widehat{Q_{L}}$ is
$\mu$-admissible as well, so we may find a lower triangular matrix
$\widehat{Q_{L}^{0}}$ such that
\[ \widehat{Q_{L}} = D_{\mu}^{-1} \widehat{Q_{L}^{0}} D_{\mu}. \]
 Since we require
 $\|(\widehat{Q_{L}^{0}})_{IH} \|=0$ to be satisfied for all $I$ and
 $H$, we may write:
\begin{align}
(QNT^{-1})_{IJ}  & = (\widehat{Q_{L}}\widehat{Q_{U}}NT^{-1})_{IJ} \nonumber \\
& = \sum_{H \subseteq I}
(\widehat{Q_{L}})_{IH} \Big((\widehat{Q_{U}}NT^{-1})_{HJ}\Big)  \nonumber \\
& = \sum_{H \subseteq I}
(\widehat{Q_{L}})_{IH} \Bigg(\sum_{S \subseteq J}(\widehat{Q_{U}}NT_{L})_{HS}(T_{U})_{SJ}\Bigg)  \label{expanded} \\
& =  \sum_{H \subseteq I} (D_{\mu}^{-1} \widehat{Q_{L}^{0}}
D_{\mu})_{IH} (\widehat{Q_{U}}NT^{-1})_{HJ}  \nonumber \\
& = \label{case}  \sum_{H \subseteq I} ( \widehat{Q_{L}^{0}} )_{IH}
(\widehat{Q_{U}}NT^{-1})_{HJ} \cdot t^{|\mu_{H}| - | \mu_{I}|}
 \end{align} So, we first
see $(QNT^{-1})_{IJ}$ expressed a sum in the form of Equation~\ref{expanded}, from which, by our previous reasoning, by a generic choice of $T_{U}$ we may ensure no catastrophic cancelation has occurred in the sum.  But then, the same terms appearing in the right side if Equation~\ref{expanded} are re-expressed in Equation~\ref{case}, and in this form we may conclude that Equation~\ref{second} may be satisfied.

In order to show that we may satisfy Equation~\ref{third},  we may
write
\begin{align}
 N^*_{IJ}  =  (QNT_{L}T_{U})_{IJ}  & =
 \sum_{H \subseteq J} (QNT_{L})_{IH}(T_{U})_{HJ} ,
\end{align} Again, since the minors $(T_{U})_{HJ}$ are uncoupled to
the other terms, we may ensure there is not catastrophic
cancelation, so that Equation~\ref{third} may be satisfied.
\end{proof}

\begin{df}
 Let us call a matrix pair $(D_{\mu}, N^*)$ a {\em $\mu$-generic} matrix pair associated to $N \in M_{r}(R)$ with respect to a
partition $\mu$ if we can factor $N^*$ as \[ N^* = QNT^{-1},
\] where $Q= Q_{L}Q_{U} = D_{\mu}^{-1}Q_{L}^{0}D_{\mu}Q_{U}$ is
$\mu$-admissible, $Q_{L}^{0}, Q_{U}$ are lower and upper triangular,
respectively, and $Q_{L},Q_{U}$ and $T^{-1}=T_{L}T_{U}$ satisfy
Equations~\ref{first},~\ref{second} and~\ref{third} of
Lemma~\ref{det-lem}.
  We shall simply say $N^*$ is {\em $\mu$-generic} if $N^*$ is a
$\mu$-generic matrix associated to some $N \in M_{r}(R)$.
\end{df}

It is from the $\mu$-generic $N^*$ that we will determine a
Littlewood-Richardson filling, and this matrix form appears to be of
some independent interest.  Before proceeding, we will require the
following technical result, which underpins the combinatorial
structure of our results.

\begin{prop}
Suppose $N^*$ is $\mu$-generic with respect to a matrix $N$. Then if
$I$ and $J$ are index sets of length $k$, for $k \leq r$, and $I
\subseteq H \subseteq J$, we have: \label{det-gap}
\begin{align}
\| N^*_{IJ} \| &\leq \| N^*_{HJ} \| \leq \| N^*_{IJ} \| +
|\mu_{I}|-|\mu_{H}|, \label{det-ineq}
\end{align}
and \begin{equation}  \| N^*_{IH} \| \geq \| N^*_{IJ} \|.
\label{col-ineq} \end{equation}
\end{prop}

\begin{proof}
  Suppose $I \subseteq H \subseteq J$, so
that, in particular, $i_t \leq h_t \leq j_t$, for $1 \leq t \leq k$.
Let us use the $\mu$-generic condition and factor $N^*$ as:
$N^*=QNT^{-1}$, where $Q= Q_{L}Q_{U} = D_{\mu}^{-1}Q_{L}^{0}Q_{U} =
\widehat{Q}_{U}\widehat{Q}_{L}$ is $\mu$-admissible. Then, by the
Cauchy-Binet formula

\begin{align*}
\left\| N^{*}_{IJ} \right\| &= \left\|
(QNT^{-1})_{IJ} \right\| \\
&= \min_{I \subseteq S} \left\{ \left\| (Q_{L}NT^{-1})_{SJ} \right\|
\right\}, \qquad \qquad \qquad
\text{By Eq.~\ref{first} of Lemma~\ref{det-lem}}\\
& \leq \min_{H \subseteq S} \left\{ \left\| (Q_{L}NT^{-1})_{SJ}
\right\| \right\},  \qquad \qquad \qquad
\text{(since $I \subseteq H$)}\\
&= \left\| N^{*}_{HJ} \right\|.
\end{align*}

For the second inequality, we have:

\begin{align*} \|N^*_{HJ}\|& = \left\|
(QNT)_{HJ}  \right\|  \\
& = \min_{S\subseteq H} \left\{ |\mu_{S}| - | \mu_{H} |  + \left\|(
Q_{U}NT^{-1})_{SJ} \right\| \right\}, \qquad
\text{By Eq.\ref{second} of Lemma~\ref{det-lem}}\\
& \leq \min_{S\subseteq I} \left\{ |\mu_{S}| - | \mu_{H} |
+ \left\|( Q_{U}NT^{-1})_{SJ} \right\| \right\} , \quad \text{(since $I
\subseteq H$)} \\
& =\min_{S\subseteq I} \left\{ |\mu_{S}| - |\mu_{I}|  + \left\|( Q_{U}NT^{-1})_{SJ} \right\| \right\} + |\mu_{I}| -| \mu_{H} | \\
& = \left\| N^{*}_{IJ} \right\| +|\mu_{I}| -| \mu_{H} | .
\end{align*}

Lastly, for Inequality~\ref{col-ineq}, we have
\begin{align*} \| N_{IH}^* \| & = \| (QNT_{L}T_{U})_{IH} \| \\
& = \min_{S \subseteq H} \| (QNT_{L})_{IS} \| \\
& \geq \min_{S \subseteq J} \| (QNT_{L})_{IS}  \|, \quad
\text{(since $H \subseteq J$,)} \\
& = \| N_{IJ}^* \|.
\end{align*}
\end{proof}

The following corollary shows how we may easily determine which rows of a
$\mu$-generic matrix may be used to compute its invariant factors.

\begin{cor}  Suppose $N^*$ is a $\mu$-generic matrix such that
$inv\,(N^*)=(\nu_{1} \geq \nu_{2} \geq \dots \geq \nu_{r})$.  Then,
if $I_{s} = (1,2, \dots, s)$, $H_{(r-s)} = ((r-s+1), (r-s+2), \dots,
r)$, we have
\[ \| N_{I_{s}H_{(r-s)}}^* \| =
\nu_{r-s+1}+ \nu_{r-s+2} + \dots + \nu_{r} \quad
\hbox{and} \quad \| (D_{\mu}N^*)_{H_{(r-s)}H_{(r-s)}} \| =
\lambda_{r-s+1}+\lambda_{r-s+2}+\dots + \lambda_r,
\] where $\lambda=(\lambda_{1}, \dots, \lambda_r) = inv\,(D_{\mu}N^*)$.\label{row-inv}
\end{cor}
\begin{proof} By construction, $I_{s} \subseteq I$ for any other index set $I$ of length
$s$.  Thus, by the first inequality appearing
Inequality~\ref{det-ineq} and Inequality~\ref{col-ineq} of
Proposition~\ref{det-gap}, $\|N_{I_{s}H_{(r-s)}} \|$, appearing in
the upper left corner, is minimal among the orders of all $s \times
s$ minors of $N^*$, so this order must be the sum of the smallest
$s$ invariant factors, from which the result follows.  The second
equality follows by noting that by right side of
Inequality~\ref{det-ineq} of Proposition~\ref{det-gap},  the orders
of minors of $D_{\mu}N^*$ must {\em increase} row index sets
decrease, so that now the {\em bottom} $s$ rows of $D_{\mu}N^*$ now
correspond to the smallest $s$ invariants, just as the {\em top} $s$
rows of $N^*$ did in the previous case.
\end{proof}

\section{Littlewood-Richardson Fillings From $\mu$-generic Matrix Pairs}

In this section we will show how to determine from a pair $(M,N)$ a
\LR\ filling of $\lambda / \mu$ with content $\nu$, when $inv\,(M)=
\mu$, $inv\,(N)= \nu$, and $inv\,(MN) = \lambda$.

\begin{df}
Suppose that $(D_{\mu},N^*)$ is a fixed $\mu$-generic pair in the
orbit of the given pair $(M,N) \in {\cal M}_{r}^{(2)}$. Let the
symbols
\[\left\| \Bigl( i_{1}, \ldots , i_s \Bigr) \right\|, \quad \left\| \Bigl( (i_{1})^{\wedge}, (i_2)^{\wedge}, \ldots , (i_k)^{\wedge}\Bigr) \right\|
 \] denote, respectively, the order of the minor
 of the $\mu$-generic matrix $N^*$ with rows $i_1, \dots, i_s$, and
 the right-most distinct columns possible.
Secondly, when using the ``\,$^{\wedge}$" symbol, the order of the
minor of $N^*$ whose rows include all rows $1$ through $r$ but with
the rows $i_1, i_2, \ldots, i_k$ omitted, again using the right-most
columns resulting in a square submatrix.
\end{df}

 We will omit the dependence of the above notation on the fixed $\mu$-generic
matrix $N^*$.

\begin{thm}
Let $(M,N) \in {\cal M}_{r}^{(2)}$ and suppose that $N^*$ is a $\mu$-generic
matrix associated to $N$. Let us define a triangular array of integers $\{
k_{ij} \}$, for $1 \leq i \leq r$, and $i \leq j \leq r$, by declaring
\begin{equation}
k_{1j} + k_{2j} + \dots + k_{ij} =\left\| \Bigl( (j-i)^{\wedge},
 \ldots ,(j-1)^{\wedge} \Bigr) \right\| - \left\| \Bigl(
(j-i+1)^{\wedge}, \ldots ,(j)^{\wedge} \Bigr)
\right\| .\\
\label{seq-def}
\end{equation}
Then, the set $F = \{ k_{ij} : 1 \leq i \leq r, \ i \leq j \leq r \}$ is a \LR\
filling of the skew shape $\lambda / \mu$ with content $\nu$, where
$inv\,(M)=\mu$, $inv\,(N) = \nu$, and $inv\,(MN)= \lambda$.  Equivalently,
setting
\begin{equation} \lambda_{j}^{(i)} = \mu_{j}+k_{1j} +
\dots + k_{ij} \label{lrseq} \end{equation}
 defines a \LR\ sequence of
type $(\mu,\nu; \lambda)$.\label{kij-lr}
\end{thm}

The formula in Equation~\ref{seq-def} allows us to define the size of row $j$
in the partition $\lambda^{(i)}$ of (what we shall prove to be) a \LR\
sequence. Since our notation for omitted indices in determinants is only to be used when removing a non-empty increasing sequence of indices, we will adopt the convention that
\[ \| (p)^{\wedge} , \dots, (q)^{\wedge} \| = \| 1, 2, \dots, r \| \qquad \hbox{if $p>q$.} \]  With this, we can use Equation~\ref{seq-def} above to
define the individual entries $k_{ij}$, according to the formula

\begin{multline}
k_{ij} =\left\| \Bigl( (j-i)^{\wedge}, (j-i+1)^{\wedge}, \ldots ,
(j-1)^{\wedge} \Bigr) \right\| - \left\| \Bigl( (j-i+1)^{\wedge}, \ldots
,(j)^{\wedge} \Bigr)
\right\| \\
- \Biggl( \left\| \Bigl( (j-i+1)^{\wedge}, (j-i+2)^{\wedge}, \ldots ,
(j-1)^{\wedge} \Bigr) \right\| - \left\| \Bigl( (j-i+2)^{\wedge}, \ldots ,
(j)^{\wedge} \Bigr) \right\| \Biggr).\label{kij-def}
\end{multline}

Note that all the determinants above have the same form.  Namely, they all have
a single, consecutive sequence of rows removed.  We can actually give a
combinatorial meaning to the orders of these determinants. For example, suppose $r=5$, and let us arrange the
integers in a \LR\ filling in a triangular array:
\[ \begin{array}{ccccc} k_{11} & & & & \\ k_{12}&k_{22}& & & \\
k_{13}& k_{23} & k_{33} & & \\k_{14} &k_{24} & k_{34} & k_{44} & \\ k_{15} &
k_{25}& k_{35} & k_{45} & k_{55} .
\end{array} \]

Our interpretation will tell us how to remove terms from this array, so that
the order of our determinant equals the sum of the remaining terms.  For
example, in the determinant
\[ \left\| \Bigl( (4-2)^{\wedge}, (4-1)^{\wedge}
 \Bigr) \right\| = \left\| \Bigl( (2)^{\wedge}, (3)^{\wedge} \Bigr) \right\|, \]
we will read from the right to the left, so we begin with the omitted row 3.
This will denote that we first {\em remove} the $k_{ij}$'s appearing in the
first {\em three} rows of the array, starting in the first row.  The next $2$
will then denote that we {\em remove} the $k_{ij}$'s appearing in the first
{\em two} rows in which they appear (that is, starting in the second row).
Thus, the array associated to the determinant above is:

\[ \left\| \Bigl(
 (2)^{\wedge}, (3)^{\wedge} \Bigr)
\right\| \Rightarrow  \begin{array}{ccccc} \begin{array}{c} (3)^{\wedge} \\
\hbox{\fbox{ $\begin{array}{c} k_{11} \\
k_{12} \\ k_{13} \end{array}$}}  \end{array}& \begin{array}{c} (2)^{\wedge} \\ \hbox{\fbox{ $\begin{array}{c}  \\
k_{22} \\ k_{23} \end{array}$}} \end{array} & \begin{array}{c} \\ \hbox{ $\begin{array}{c}  \\
\\ k_{33} \end{array}$} \end{array} & & \\
 k_{14} & k_{24}& k_{34} & k_{44} & \\ k_{15} & k_{25} & k_{35} & k_{45} & k_{55} \end{array}  \quad
=  \begin{array}{ccccc}\_ & & & & \\ \_&\_& & & \\
\_& \_ & k_{33} & & \\k_{14} &k_{24} & k_{34} & k_{44} & \\ k_{15} & k_{25}&
k_{35} & k_{45} & k_{55}
\end{array}\]

We claim (and will subsequently show), that the order of the determinant $\left\| \Bigl( (2)^{\wedge}, (3)^{\wedge}\Bigr)\right\|$ equals the sum of the $k_{ij}$'s in the right-hand side of the above picture, where (as we shall also show), the integers so defined form a \LR\ filling of $\lambda / \mu$ with content $\nu$.

 Similarly, in the determinant in which we omit rows 3 and 4 we associate the
array
\[ \left\| \Bigl(
 (3)^{\wedge}, (4)^{\wedge} \Bigr)
\right\| \Rightarrow  \begin{array}{ccccc}
 \begin{array}{c} (4)^{\wedge} \\
\hbox{\fbox{ $\begin{array}{c} k_{11} \\
k_{12} \\ k_{13} \\ k_{14} \end{array}$}}  \end{array}
&
\begin{array}{c} (3)^{\wedge} \\ \hbox{\fbox{ $\begin{array}{c}  \\
k_{22} \\ k_{23} \\ k_{24} \end{array}$}} \end{array}
&
\begin{array}{c} \\ \hbox{ $\begin{array}{c}  \\
\\ k_{33} \\ k_{34} \end{array}$} \end{array}
& \hbox{ $\begin{array}{c}  \\ \\
\\  \\ k_{44} \end{array}$}
& \\
 k_{15} & k_{25} & k_{35} & k_{45} & k_{55}
\end{array}  \quad
=  \begin{array}{ccccc}\_ & & & & \\ \_&\_& & & \\
\_& \_ & k_{33} & & \\\_ &\_ & k_{34} & k_{44} & \\ k_{15} & k_{25}& k_{35} &
k_{45} & k_{55}
\end{array}\]

Consequently, we associate to the difference of orders of determinants the
array: \begin{align*} \left\| \Bigl( (2)^{\wedge},(3)^{\wedge} \Bigr) \right\|
- \left\| \Bigl(
 (3)^{\wedge},(4)^{\wedge} \Bigr)
\right\| &     \Rightarrow     \begin{array}{ccccc} \_ & & & & \\ \_
& \_ & & & \\
\_ & \_&k_{33} & & \\ k_{14} & k_{24}& k_{34} & k_{44} & \\ k_{15} & k_{25} &
k_{35}& k_{45} & k_{55}
\end{array} -
\quad \begin{array}{ccccc} \_ & & & & \\ \_ & \_ & & & \\
\_ & \_& k_{33} & & \\ \_ & \_& k_{34} & k_{44} & \\ k_{15} & k_{25} & k_{35} & k_{45} & k_{55} \end{array}      \\
& = \begin{array}{ccccc} \_ & & & & \\ \_ & \_ & & & \\
\_ & \_& \_ & & \\ k_{14}&k_{24}& \_&\_ & \\ \_ & \_ & \_ & \_ & \_ \end{array}
\quad = k_{14}+k_{24},
\end{align*}
which is just the form of Equation~\ref{seq-def} when $r=5$, $j=4$, and $i=2$.

\medskip

The study of the structure of \LR\ fillings in the form of the integers $\{
k_{ij}\}$ can be found in a variety of contexts (see \cite{knut} and
\cite{pak}).  What we find is that these fillings do more than {\em count},
they {\em explain} how the invariant factors of one matrix are distributed with
respect to another.

\noindent Let us now show these interpretations are justified by proving
Theorem~\ref{kij-lr}.

\medskip

\begin{proof}
We shall, in turn, prove (LR1), (LR2), (LR3), and (LR4) of
Definition~\ref{LR-def} for the set of integers $F= \{ k_{ij} \}$
defined by Equation~\ref{kij-def}.

\medskip

\noindent {\bf(LR1)}  We need to show:
\begin{eqnarray}
\sum_{s=1}^{j} k_{sj}& =& \lambda_{j} - \mu_{j}, \quad 1 \leq j \leq r\label{lambda-sum}\\
\sum_{s=i}^{r}k_{is}& = & \nu_{i}.\quad 1 \leq i \leq r
\label{nu-sum}\end{eqnarray}

Using Equation~\ref{lrseq} we see Equation~\ref{lambda-sum} is just the
requirement that $\lambda^{(r)} = \lambda = inv\,(MN)$.  We claim it will be
sufficient to prove:
\begin{equation} \left\| \Bigl( (j-r)^{\wedge}, (j-r+1)^{\wedge}, \ldots ,
(j-2)^{\wedge},(j-1)^{\wedge} \Bigr) \right\| = (\lambda_{j} -\mu_{j}) +
(\lambda_{j+1} - \mu_{j+1}) + \cdots + (\lambda_{r}-\mu_{r}),
\label{lam-sum}\end{equation} for $j=1, \ldots , r$. This is because, on the
one hand, the right side of Equation~\ref{lambda-sum} is obtained by taking the
difference of the right sides of Equation~\ref{lam-sum}, first using $j$ as
above, and then replacing $j$ with $j+1$, but then, on the other hand, noting
that the corresponding differences on the left side of Equation~\ref{lam-sum}
gives us the right side of Equation~\ref{seq-def}, from which the result
follows. The claim in Equation~\ref{lam-sum}, however, follows from the second
part of Corollary~\ref{row-inv}.

Before proving Equation~\ref{nu-sum}, we will need here (and later) the following
lemma, which is really just a consequence of the telescoping of terms first
noted in our main example.  If Equation~\ref{seq-def} shows us how to compute the sum of the $k_{pq}$'s along a given {\em row} of a proposed \LR\ filling, then the following lemma shows how to compute a {\em block} of $k_{pq}$'s, for $1 \leq p \leq i$ and $j \leq q \leq l$, along with the sum of the $k_{iq}$'s for $i \leq q \leq j$.

\begin{lem} With $k_{ij}$ defined by Equation~\ref{kij-def} for all $1 \leq j \leq r$ and $1 \leq i \leq j$, we have
\begin{enumerate}
\item $\sum_{\beta = j}^{l}  (k_{1 \beta} + k_{2 \beta} + \cdots + k_{i \beta}) =
\| (j-i)^{\wedge} \dots (j-1)^{\wedge} \| - \| (l - i + 1)^{\wedge} \dots (l)^{\wedge} \|,$ for $j \leq l$.
\item $k_{ii} + k_{i,(i+1)} + \cdots + k_{ij} = \| (j-i+2)^{\wedge} \dots (j)^{\wedge} \| - \| (j-i+1)^{\wedge} \dots (j)^{\wedge} \|$.
    \end{enumerate}\label{word-lem}
\end{lem}

\begin{proof}
The first equality is an immediate consequence of noting the telescoping of terms in Equation~\ref{seq-def} applied to successive rows.  The second equality follows from calculating the difference between an instance of the first equality ending with $k_{ij}$,  subtracting an instance ending with $k_{(i-1),j}$, and then canceling terms (Note that when $i=1$ that these formulas still make sense, using our convention concerning the meaning of omitted indices.)
\end{proof}

\medskip

So, to prove Equation~\ref{nu-sum}, we apply the second equality in Lemma~\ref{word-lem} in the case
$j =r$, and obtain
\begin{align*}
 k_{ii}+ \dots + k_{ir} & = \| ((r-i+2)^{\wedge} , \dots , (r)^{\wedge} )\| - \|
((r-i+1)^{\wedge}, \dots , (r)^{\wedge}) \| \\
& = \| (1, \dots , (r-i+1) )\| - \| (1, \dots, (r-i) ) \| \\
&= (\nu_{i} + \cdots + \nu_{r} ) - (\nu_{(i+1)} + \dots + \nu_{r}) \\
& =\nu_i ,
\end{align*}
where the penultimate equality follows from the first part of
Corollary~\ref{row-inv}.

\medskip
The proofs for (LR2), (LR3) and (LR4) are surprisingly similar, and all depend
on computing minors of $N^*$ with explicit submatrices on which we may put the
matrix into a convenient block form from which the determinant may be computed.

\medskip

\noindent {\bf (LR2)} Let us re-write the condition $k_{ij} \geq 0$,
using
Equation~\ref{kij-def}, but expressed positively
(in terms of rows that are
kept instead of omitted), as:
\begin{multline} \left\| \Bigl( 1, \ldots , (j-i),j,\ldots, r \Bigr)
\right\|
+ \left\| \Bigl( 1, \ldots , (j-i),(j+1), \ldots ,r \Bigr) \right\| \leq \\
\left\| \Bigl( 1, \ldots , (j-i-1), j, \ldots , r \Bigr) \right\| + \left\|
\Bigl( 1, \ldots , (j-i+1), (j+1), \ldots , r \Bigr) \right\|. \label{kij-ineq}
\end{multline}

Each minor starts on the right in column $r$, and uses consecutive
columns
 as we proceed to the left.  So, for example, the first term
\[ \left\| \Bigl( 1, \ldots , (j-i),j,\ldots, r \Bigr) \right\|, \]
would use columns $i$ to $r$.  Recall that since $N^*$ is
$\mu$-generic, it is upper triangular.  By a slight abuse of
notation, we will will use our notation for the minor (a
determinant) to denote a submatrix in order to re-express the above
in block form as:

\[ \begin{bmatrix}
N^*
\begin{pmatrix} 1 \ 2 \ \dots  \ (j-i) \\
i \ (i+1) \ \dots  \ (j-1) \end{pmatrix} & N ^*
\begin{pmatrix} 1 \ 2 \ \dots  \ (j-i) \\
j \ (j+1) \ \dots  \ r \end{pmatrix} \\
0 & N^*
\begin{pmatrix} j \ (j+1) \ \dots  \ r \\
j \ (j+1) \ \dots  \ r \end{pmatrix} \end{bmatrix} ,\]

Thus, we can express Inequality~\ref{kij-ineq} in block-form as:
\begin{multline*}
\begin{Vmatrix} N^*
\begin{pmatrix} 1 \ \dots  \ (j-i) \\
i \ \dots  \ (j-1) \end{pmatrix} & N^*
\begin{pmatrix} 1  \dots  (j-i) \\
j  \dots  r \end{pmatrix}  \\
0 & N^*
\begin{pmatrix} j  \dots  r \\
j  \dots  \ r \end{pmatrix} \end{Vmatrix} + \begin{Vmatrix} N^*
\begin{pmatrix} 1 \dots   (j-i) \\
(i+1)  \dots   j \end{pmatrix} & N^*
\begin{pmatrix} 1  \dots  (j-i) \\
 (j+1)  \dots  r \end{pmatrix}  \\
0 & N^*
\begin{pmatrix} (j+1)  \dots   r \\
(j+1)  \dots   r \end{pmatrix} \end{Vmatrix} \leq \\
\qquad \begin{Vmatrix} N^*
\begin{pmatrix} 1 \dots  (j-i-1)  \\
(i+1)  \dots   (j-1) \end{pmatrix} & N^*
\begin{pmatrix} 1 \dots  (j-i-1) \\
j  \dots   r \end{pmatrix}  \\
0 & N^*
\begin{pmatrix} j  \dots  r \\
j  \dots   r \end{pmatrix} \end{Vmatrix} + \begin{Vmatrix} N^*
\begin{pmatrix} 1  \dots  (j-i+1) \\
i  \dots  j \end{pmatrix} & N^*
\begin{pmatrix} 1  \dots  (j-i+1) \\
 (j+1) \dots   r \end{pmatrix}  \\
0 & N^*
\begin{pmatrix} (j+1)  \dots   r \\
(j+1)  \dots   r \end{pmatrix} \end{Vmatrix} .
\end{multline*}

Since the orders of the determinants above are the sums of the orders of the
determinants of their block diagonals, we may cancel the orders of the
south-east blocks in the above inequality, so that it is sufficient to prove:
 \begin{multline*}
\underbrace{\begin{Vmatrix} N^*
\begin{pmatrix} 1 \ \dots  \ (j-i) \\
i \ \dots  \ (j-1) \end{pmatrix} \end{Vmatrix}}_{S} +
\underbrace{\begin{Vmatrix} N^*
\begin{pmatrix} 1 \dots   (j-i) \\
(i+1)  \dots   j \end{pmatrix} \end{Vmatrix}}_{T_{1}} \leq \\
\underbrace{\begin{Vmatrix} N^*
\begin{pmatrix} 1 \dots  (j-i-1)  \\
(i+1)  \dots   (j-1) \end{pmatrix}\end{Vmatrix}}_{S_{1}} +
\underbrace{\begin{Vmatrix} N^*
\begin{pmatrix} 1  \dots  (j-i+1) \\
i  \dots  j \end{pmatrix} \end{Vmatrix}}_{T} . \end{multline*}

Notice that by Proposition~\ref{det-gap}, a submatrix of $N^*$ will
satisfy the same determinantal inequalities as does the full matrix,
so that Corollary~\ref{row-inv} will still apply. So, since the
submatrix $S_1$ is the upper right corner of $S$, if $inv\,(S) = (
\beta_{1} \geq \beta_{2} \geq \cdots \geq \beta_{j-i} )$, then
$inv\,(S_{1}) = ( \beta_{2} \geq \beta_{2} \geq \cdots \geq
\beta_{j-i} )$. Similarly, if $inv\,(T) = ( \alpha_{1} \geq
\alpha_{2} \geq \cdots \geq \alpha_{j-i+1} )$ then $inv\,(T_1) = (
\alpha_{2} \geq \cdots \geq \alpha_{j-i+1} )$. Substituting this
into the above, we see that in order to prove
Inequality~\ref{kij-ineq} it is sufficient to establish $\beta_{1}
\leq \alpha_{1}$.  This, however, follows from noting that matrix
$S$ is a $(j-i) \times (j-i)$ submatrix of $T$ (in rows $1$ through
$(j-i)$), and hence the highest invariant factor of $S$ (that is,
$\beta_{1}$) is bounded by the highest invariant factor of $T$
(namely, $\alpha_{1}$), by the so-called ``interlacing''
inequalities of invariant factors, as found in, for instance,
\cite{carlson-sa}, and also~\cite{sa},~\cite{Thompson}.

\medskip

 {\bf (LR3)}  The column strictness condition (LR3) asserts that
 in a \LR\ filling of $\lambda / \mu$,
the sum of the number of $1$'s through $i$'s appearing in row $j$ of the skew
shape must not extend beyond the sum of the number of $1$'s through $(i-1)$'s
appearing in row $(j-1)$.  That is,
\begin{equation} k_{1j} + \cdots k_{ij} + \mu_{j} \leq k_{i,(j-1)} + \cdots
+ k_{(i-1),(j-1)} + \mu_{(j-1)}. \label{cs}
\end{equation}

The sums of the $k_{pq}$'s appearing in both sides of Inequality~\ref{cs} can
be expressed using Equation~\ref{seq-def}. As before, we can write this
inequality in terms of blocks of matrices with right-justified columns.  In
this case, we partition all the matrices appearing at row/column $(j+1)$, so
that we can cancel the orders of the determinants of these lower blocks.  Thus,
in order to prove Inequality~\ref{cs} it will be sufficient to prove

\begin{multline*}
 \underbrace{\left\| N^*
\begin{pmatrix} 1 \dots   (j-i) \ j \\
i  \dots  (j-1) \  j \end{pmatrix} \right\|}_{S} +
\underbrace{\left\| N^*
\begin{pmatrix} 1  \dots   (j-i-1) \ j \\
(i+1)   \dots  (j-1)\ j \end{pmatrix} \right\|}_{T_1} \leq \\
 \underbrace{\left\| N^*
\begin{pmatrix} 1  \dots  (j-i) \\
(i+1)  \dots  j \end{pmatrix} \right\|}_{S_{1}}  +
\underbrace{\left\| N^*
\begin{pmatrix} 1 \dots  (j-i-1) \ (j-1)\ j \\
i  \dots  (j-2)  \ (j-1) \ j \end{pmatrix} \right\|}_{T} + \mu_{(j-1)} -
\mu_{j}.
\end{multline*}

By Proposition~\ref{det-gap} we have
\[ \underbrace{\left\| N^*
\begin{pmatrix} 1  \dots   (j-i-1) \ j \\
(i+1)   \dots  (j-1)\ j \end{pmatrix} \right\|}_{T_1} \leq
\underbrace{\left\| N^*
\begin{pmatrix} 1  \dots   (j-i-1) \ (j-1) \\
(i+1)   \dots  (j-1)\ j \end{pmatrix} \right\|}_{T_{1}^*} + \mu_{(j-1)} -
\mu_{j},
\] So, by substituting into the
above we see it is sufficient to prove
\begin{multline*}
 \underbrace{\left\| N^*
\begin{pmatrix} 1 \dots   (j-i) \ j \\
i  \dots  (j-1) \  j \end{pmatrix} \right\|}_{S} +
\underbrace{\left\| N^*
\begin{pmatrix} 1  \dots   (j-i-1) \ (j-1) \\
(i+1)   \dots  (j-1)\ j \end{pmatrix} \right\|}_{T^*_1} \leq \\
 \underbrace{\left\| N^*
\begin{pmatrix} 1  \dots  (j-i) \\
(i+1)  \dots  j \end{pmatrix} \right\|}_{S_{1}}  +
\underbrace{\left\| N^*
\begin{pmatrix} 1 \dots  (j-i-1) \ (j-1)\ j \\
i  \dots  (j-2)  \ (j-1) \ j \end{pmatrix} \right\|}_{T} .
\end{multline*}

Now, as before, we see that if $inv\,(S) = ( \beta_{1} \geq \dots
\geq \beta_{(j-i+1)})$, then $ = inv\,(S_{1}) = (\beta_{2} \geq
\dots \geq \beta_{(j-i+1)})$, and also $inv\,(T) = (\alpha_{1} \geq
\dots \geq \alpha_{(j-i-1)})$, $inv\,(T_{1}^*) = (\alpha_{2} \geq
\dots \geq \alpha_{(j-i-1)})$.
  So, it is sufficient to prove $\alpha_{1} \leq \beta_{1}$, but
this follows from the interlacing inequalities.

\medskip

{\bf (LR4)} The word condition (LR4) may be translated, using
Lemma~\ref{word-lem},  into the requirement:
\begin{multline*} \| ( (j-i+2)^{\wedge},\dots, (j+1)^{\wedge} ) \| - \|
((j-i+1)^{\wedge} ,\dots, (j+1)^{\wedge} ) \| \leq \| ((j-i+2)^{\wedge}, \dots ,
(j)^{\wedge})\| - \| ( (j-i+1)^{\wedge}, (j)^{\wedge}) \|,
\end{multline*} which, written positively, becomes
\begin{multline*}   \|(1 \dots,
(j-i), (j+1), \dots , r) \| +\| (1, \dots, (j-i+1), (j+2), \dots, r ) \|\leq \\
\| ( 1, \dots, (j-i), (j+2), \dots r) \| + \|( 1, \dots (j-i+1) ,(j+1), \dots, r)
\|.
\end{multline*}

As in the conditions (LR2) and (LR3), we write these matrices in block form,
clearing to the left from column $(j+2)$.  We may then cancel the
determinants of the blocks in the lower corners, so that it is sufficient to
show
\begin{multline*}
\underbrace{\left\| N^*
\begin{pmatrix} 1\dots    (j-i),(j+1) \\
(i+1)  \dots  (j+1)  \end{pmatrix} \right\|}_{S}
+\underbrace{\left\| N^*
\begin{pmatrix} 1  \dots   (j-i+1)  \\
(i+1)   \dots  (j+1)\end{pmatrix} \right\|}_{T_{1}} \leq \\
 \underbrace{\left\| N^*
\begin{pmatrix} 1 \dots   \ (j-i) \\
(i+2)  \dots  (j+1) \end{pmatrix} \right\|}_{S_1} +
\underbrace{\left\| N^*
\begin{pmatrix} 1  \dots  (j-i+1),(j+1)  \\
i  \dots  (j+1) \end{pmatrix} \right\|}_{T} . \end{multline*}

As before we see that if $inv\,(T) = (\beta_{1} \geq \dots \geq
\beta_{j-i+2})$, then $inv\,(T_{1})=(\beta_{2} \leq \dots
\beta_{j-i+2})$ since $T_{1}$ is the upper right corner of $T$.
Similarly, $inv\,(S) = (\alpha_1 \geq \dots \geq \alpha_{j-i+1})$
and $inv\,(S_{1}) = (\alpha_{2} \geq \dots \geq \alpha_{j-i+1})$.
Thus, it only remains to prove $\alpha_{1} \leq \beta_{1}$, which
follows from the interlacing inequalities.

\end{proof}

\section{Uniqueness}

In this Section we shall prove that the \LR\ filling associated to a matrix
pair $(M,N)$ is an invariant of the orbit under pair equivalence.  We will do
so by showing that the orders of minors of $\mu$-generic matrices associated to
a matrix $N$ are an invariant of the orbit of $N$.  Before proceeding, we will
need the following technical lemma.

\begin{lem}  Let $Q$ be an $ r \times r$ $\mu$-admissible matrix, and let $I$ and $H$ be index sets of length $k \leq r$.  Define the index set
$Min(I,H)$ by
\[ Min(I,H) = (m_1, \ldots , m_{k}), \quad m_{s} = \min \{ i_s, h_s \}. \]
then
\[| \mu_{Min(I,H)}| - |\mu_{I}| \leq  \| Q_{IH} \|. \] \label{min-lem}
\end{lem}
\begin{proof}

First note that, by Proposition~\ref{orbit}, the $IH$ minor of the
$\mu$-admissible matrix $Q$ may be factored $Q_{IH} =
(D_{\mu}^{-1})_{II}Q^{0}_{IH}(D_{\mu})_{HH}$ where $Q^{0}$ is an invertible
matrix. Then,
\[  \| Q_{IH} \| = \| (D_{\mu}^{-1})_{II}Q^{0}_{IH}(D_{\mu})_{HH} \| =
|\mu_{H}| - | \mu_{I}|+ \|  Q^{0}_{IH} \|. \]  Thus, to prove
\[ |\mu_{Min (I,H )}| - |\mu_{I}| \leq \| Q_{IH} \| = |\mu_{H} |- |\mu_{I} |+\| Q^{0}_{IH}
\| \]
 it is sufficient to verify
\begin{equation}
 |\mu_{Min( I,H)}| - |\mu_{H}| \leq \| Q^{0}_{IH} \| .
\label{claim}
\end{equation}

 We will prove Equation~\ref{claim} by induction on $k$, the size of the minor
$Q^{0}_{IH}$. For the base case $k=1$, we may assume $Q_{IH} =
q_{ih}$ for indices $i$ and $h$, so that
\[ Q_{IH} = q_{ih} =t^{-\mu_{i}}q^{0}_{ih}t^{\mu_{h}}. \] Since $Q$ is defined over $R$, we have
\[ 0 \leq \| q_{ih} \| = \mu_{h} - \mu_{i} + \| q^{0}_{ih} \| ,  \]
so that if $\min \{ i,h \} = i$, then
\[ |\mu_{Min( I,H )} |- |\mu_{H}| =
\mu_{\min \{ i,h \} } - \mu_{h} \leq \mu_{i} - \mu_{h} \leq \|q^{0}_{ih} \| =
\| Q^{(0)}_{IH} \| .
\]

If, however, $\min \{ i,h \} = h$, then
\[ |\mu_{Min( I,H )}|- |\mu_{H} |= \mu_{\min \{ i,h \}} - \mu_{h} = 0 \leq
 \|q^{0}_{ih} \| = \|Q_{IH}^{0} \|, \]
so the base case is established.

For the general case, we expand the determinant of $Q^{0}_{IH}$ along the top
row.  Note that if $i_{1} \leq h_{1}$, then $0 \leq \mu_{i_{1}} - \mu_{h_{1}}
$, so that $0 \leq \mu_{i_{1}} - \mu_{h_{1}} \leq \mu_{i_{1}} - \mu_{h_{s}}$
for all $s \geq 1$.  Each element $q_{i_{1},h_{s}}$ along the top row of
$Q_{IH}^{0}$ satisfies
\[ 0 \leq  \mu_{h_{s}} - \mu_{i_{1}} + \|q^{0}_{i_{1},h_{s}} \|, \] so that

\[ \mu_{i_{1}} - \mu_{h_{s}}   \leq  \|q^{0}_{i_{1},h_{s}} \| .\]  But then, if
 $i_{1} \leq h_{1}$, we have

\[ \mu_{\min \{i_1,h_1 \}} - \mu_{h_{1}} \leq \mu_{i_{1}} -
\mu_{h_{1}} \leq \mu_{i_{1}} - \mu_{h_{s}} \leq \| q_{i_{1},h_{s}}^{0} \|. \] for
all $s \geq 1$. If, however, $h_{1} \leq i_{1}$, then we clearly have
$\mu_{\min \{i_1,h_1 \}}- \mu_{h_1} = 0 \leq \| q^{0}_{i_{1},h_{s}} \|$ as
well.

Thus, in expanding the determinant of $Q^{0}_{IH}$ along the top row, each
entry in this row has order at least $\mu_{\min \{i_1,h_1 \}} - \mu_{h_{1}}$.
By induction, we may assume the order of each $(k-1) \times (k-1)$ minor in the
expansion of $ \|Q^{0}_{IH} \|$ along the top row has order at least \[
\mu_{\min \{i_{2},h_{2} \}} + \dots + \mu_{\min \{ i_{k},h_{k} \}}-
(\mu_{h_{2}} + \cdots +   \mu_{h_{k}}).
\] By summing these orders, the lemma follows.

\end{proof}

\begin{prop}
Let $I$ and $J$ be index sets of length $k$, and let $N$ and $\widehat{N}$ be
$r \times r$ $\mu$-generic matrices.  Suppose
there exist $\mu$-admissible matrices $Q$ and $T$ such that
\[ QNT^{-1} = \widehat{N} . \]
Then \label{big-prop}
\[ \|N_{IJ}\| = \|\widehat{N}_{IJ}\|.\]
\end{prop}

\begin{proof}

Let us simplify notation by setting $S = T^{-1}$.  Since $\widehat{N}=QNS$, by the Cauchy-Binet formula  we have
\begin{eqnarray*}
\widehat{N}_{IJ} &= &  \sum_{H , L} Q_{IH}N_{HL}S_{LJ} .
\end{eqnarray*}

Then we have
\begin{eqnarray*} \min_{H,L} \Big\{ \|Q_{IH}N_{HL}S_{LJ} \|  \Big\} &\leq & \Bigl\| \sum_{H , L} Q_{IH}N_{HL}S_{LJ} \Bigr\| \\
& = & \| \widehat{N}_{IJ} \| .  \\
\end{eqnarray*}

We shall show that each such term $Q_{IH}N_{HL}S_{LJ}$ in the sum above has
order at least $\| N_{IJ}\|$, that is, we claim:

\[  \| N_{IJ} \|  \leq \| Q_{IH}N_{HL}S_{LJ} \|
  . \]

  To see this, note that
\begin{align*}
\| N_{IJ} \| & \leq  \| N_{Min(I,H ),J } \| + \mu_{Min( I,H )}
-\mu_{I},& \text{Prop.~\ref{det-gap},  $Min(I,H ) \subseteq I$} \\
& \leq  \| N_{Min(I,H ),J } \| +\| Q_{IH} \|,& \text{by Lemma~\ref{min-lem}} \\
& \leq  \|  N_{ HJ  } \| +  \| Q_{IH} \|, &\text{Prop.~\ref{det-gap},
 $Min(I,H ) \subseteq H$}\\
& \leq  \|  N_{ HJ  } \| +  \| Q_{IH} \| + \|S_{LJ} \| = \|
Q_{IH}N_{HL}S_{LJ} \| .
\end{align*}

Since $\widehat{N}_{IJ}$ is a sum of terms of the form $Q_{IH}N_{HL}S_{LJ}$, we
have, using the above, that
\[ \| N_{IJ} \| \leq  \Bigl\|\sum_{H , L} Q_{IH}N_{HL}S_{LJ} \Bigr\| = \| \widehat{N}_{IJ} \|. \]

However, since the hypotheses on $N$ and $\widehat{N}$ are symmetric, we
conclude also that
\[ \|N_{IJ}\| \geq \|\widehat{N}_{IJ}\|, \]
and so, finally, we have
\[ \|N_{IJ}\| = \| \widehat{N}_{IJ}\|. \]
\end{proof}
\medskip

\begin{thm}[Uniqueness]If $(M,N)$ is pair equivalent to $(M', N')$, then the
Littlewood-Richardson fillings determined by both pairs are the same.  That is,
pairs in the same $GL_{r}(R)^3$ orbit yield identical \LR\ fillings.
\end{thm}
\begin{proof}
The \LR\ filling associated to any $\mu$-generic matrix $N^*$ in the orbit of
$N$ is determined by the orders of quotients of its determinants.  By the
previous proposition, these orders are an invariant of the orbit of $N^*$, so
the result follows.
\end{proof}

\noindent {\bf Not a Complete Invariant.} Though the Littlewood-Richardson
filling determined by a pair $(M,N)$ is an invariant of the orbit, there do
exist pairs $(M,N)$ and $(M',N')$ such that both have the same filling, yet
they are not in the same orbit.  It seems that the Littlewood-Richardson
filling yields a ``discrete invariant'' of the orbit, while not uniquely
characterizing it.  A complete invariant seems to depend also on some
continuously parameterized data.  As an example, let
\[ M = M' = D_{\mu} = \left[ \begin{array}{ccc} t^6 & 0 & 0 \\ 0 & t^3 & 0 \\
0 & 0 & t^1 \end{array} \right], \] and suppose \[ N =  \left[
\begin{array}{ccc} t^8 & t^7 & t^4 \\ 0 & t^9 & 2t^6 \\ 0 & 0 & t^7 \end{array}
\right], \quad \hbox{and} \quad N'= \left[ \begin{array}{ccc} t^8 & t^7 & t^4
\\ 0 & t^9 & 4t^6
\\ 0 & 0 & 3t^7 \end{array} \right]. \]

Both of the pairs $(D_{\mu}, N)$ and $(D_{\mu}, N')$ satisfy the inequalities
of Proposition~\ref{det-gap} so that we may (using right-justified columns)
quickly see that both pairs yield the same \LR\ filling.

However, the pairs $(D_{\mu}, N)$ and $(D_{\mu}, N')$ are not in the same
orbit. If they were, there would be invertible matrices $P$, $Q$ and $T$ such
that $(P,Q,T)\cdot (D_{\mu},N) = (PD_{\mu}Q^{-1},QNT^{-1}) = (D_{\mu},N')$, so
that, in particular $Q=D_{\mu}^{-1}PD_{\mu}$ for invertible $P$ and hence
entries $q_{ij}$ in $Q$ have order at least $\mu_{j}-\mu_{i}$ whenever $j < i$.
We shall express this by writing entries of $Q$ below the diagonal in the form
$t^{\mu_{j}-\mu_{i} }q_{ij}$ for some $q_{ij} \in R$.  If we express the
$(i,j)$ entry in $T$ as $y_{ij}$, then we can re-write the equation
$QNT^{-1}=N'$ as $QN = N'T$. So $(D_{\mu},N)$ and $(D_{\mu},N')$ are in the
same orbit if and only if we can find entries $y_{ij}$ and $q_{ij}$ over the
ring $R$ such that we may solve:

\[  \left[
\begin{array}{ccc} q_{11} & q_{12} & q_{13} \\ t^3 q_{21} & q_{22} &
q_{23} \\ t^5 q_{31} & t^2 q_{32} & q_{33} \end{array} \right]\left[
\begin{array}{ccc} t^8 & t^7 & t^4 \\ 0 & t^9 & 2t^6
\\ 0 & 0 & t^7 \end{array} \right] =  \left[ \begin{array}{ccc}t^8 & t^7 & t^4 \\ 0 & t^9 &
4t^6 \\ 0 & 0 & 3t^7 \end{array} \right] \left[ \begin{array}{ccc} y_{11} &
y_{12} & y_{13} \\  y_{21} & y_{22} & y_{23} \\ y_{31} &  y_{32} & y_{33}
\end{array} \right]. \]

However, it is not possible to find a matrix $T=(y_{ij})$ defined over $R$ that
satisfies the above.  This follows by calculating $T=(N')^{-1}  QN$ in the
above form and noting first that the determinant of $T$ is a polynomial in the
uniformizing parameter $t$ with coefficients in $R$ whose constant term is the
product $q_{11}q_{22}q_{33}$.  In order for $T$ to be invertible, this product
must be a unit in $R$ (this just expresses the obvious condition that the
diagonal of the $\mu$-admissible matrix $Q$ must be composed of units).  We can
express all the entries of $(N')^{-1}Q  N$ as rational functions of $t$ with
coefficients in $R$.  For instance, the $(1,2)$ entry is
\[ \frac{(q_{11}-q_{22}) + q_{12}t^2-q_{21}t+q_{31}t^2+q_{32}t}{t} = \frac{q_{11}-q_{22}}{t}
+(q_{32}-q_{21}) + (q_{12}-q_{31})t. \]  Thus, in order for this entry to be
defined over $R$, the difference $q_{11}-q_{22}$ must have order at least 1.
That is,
\[ \| q_{11} - q_{22} \| \geq 1. \]
In order for this to happen, there must be catastrophic cancelation in the
units $q_{11}$ and $q_{22}$, so that
\[ c_{*}(q_{11}) - c_{*}(q_{22})=0. \]
Similarly, in considering also the $(1,3)$ and $(2,3)$ entries, the following
relations among the images in the residue field of $R$ among the units
$q_{11},q_{22}$ and $q_{33}$ are also necessary:
\begin{align*}
c_{*}(q_{11}) -c_{*}(q_{22})\qquad \qquad \ \  &=0 \\
c_{*}(q_{11})-2c_{*}(q_{22}) +c_{*}(q_{33})& = 0\\
-c_{*}(q_{22}) +2c_{*}(q_{33}) &=0
\end{align*}
A quick inspection, however, reveals this linear system has no non-trivial
solution.


\begin{thebibliography}{99}

\bibitem{me} G.\ Appleby, ``A Simple Approach to Matrix Realizations for
Little\-wood-Richards\-on Sequences'', {\em Linear Algebra and Its
Applications}, vol.\ 291, pp. 1-14, (1999).

\bibitem{new} G.\ Appleby and T.\ Whitehead, ``Symmetries of Hives, Generalized
Littlewood Richardson Fillings and Invariants of Matrix Pairs over Valuation
Rings,'' in preparation.

\bibitem{at-mac} M.F.\ Atiyah and I.G.\ MacDonald, {\em Introduction
to Commutative Algebra}, Reading, Massachusetts, Addison-Wesley
Publishing Company, (1969).

\bibitem{AS} O.\ Azenhas and E.\ Marques de Sa', ``Matrix Realizations
of Littlewood-Richardson Sequences'', {\em Linear and Multilinear Algebra}, 27,
pp. 229-242, (1990).

\bibitem{Az} O.\ Azenhas, ``Opposite Littlewood-Richardson Sequences and
their Matrix Realizations'', {\em Linear Algebra and Its Applications}, 225,
pp. 91-116, (1995).

\bibitem{Sottile} G.\ Benkart, F.\ Sottile and J.\ Stroomer, ``Tableau
Switching: Algorithms and Applications", {\em J. Combin. Th. Ser.
A}., 76, pp. 11-43, (1996).

\bibitem{bour-ca} N.\ Bourbaki, {\em Commutative Algebra, Chapters
1-7}, New York, Springer-Verlag, (1989).

\bibitem{carlson-sa} D.\ Carlson and E.\ Marques De Sa, ``Generalized Minimax
and Interlacing Theorems'', {\em Linear and Multilinear Algebra}, 15, pp.\
77-103, (1984).

\bibitem{fulton} W.\ Fulton, ``Eigenvalues, invariant factors, highest weights,
and Schubert calculus,'' {\em Bull. Amer.\ Math.\ Soc. } 37, pp.\
209-249, (2000).

\bibitem{fulton-book} W.\ Fulton, {\em Young Tableaux:  With Applications to
Representation Theory and Geometry}, London Mathematical Sociey
Student Texts, Cambridge, Cambridge University Press, (1997).

\bibitem{Gant} F.\ Gantmacher, {\em The Theory of Matrices,} New York,
Chelsea Press, (1959).

\bibitem{hartley} B.\ Hartley and T.O.\ Hawkes, {\em Rings, Modules
and Linear Algebra}, New York, Chapman and Hall, (1983).



\bibitem{klein} T.\ Klein, ``The multiplication of Schur functions and
extension of $p$-modules'' {\em J. London Math. Soc.}, 43, pp. 280-284, (1968).

\bibitem{knut} A.\ Knutson and T.\ Tao, The honeycomb model of
$GL_{N}(\mathbb{C})$ tensor products I: Proof of the saturation conjecture,
{\em J.\ Amer.\ Math.\ Soc.} 12, (1055-1090, (1999).



\bibitem{mac} I.G.\ Macdonald, {\em Symmetric Functions and Hall
Polynomials}, Oxford Univ. Press, London/New York, (1979).

\bibitem{pak} I.\ Pak and E.\ Vallejo, Combinatorics and geometry of
Littlewood-Richardson Cones, {\em Europ.\ J.\ Combinatorics}, 26, 995-1--8,
(2005).

\bibitem{pak-2} I.\ Pak and E.\ Vallejo, ``Reductons of Young Tableau
Bijections,'' to appear: {\em Siam J.\ Discrete Mathematics}.

\bibitem{sa} E.\ Marques de S\'a, ``Imbedding Conditions for
$\lambda$-matrices,'' {\em Linear Algebra and its Applications}, 24, pp.\
33-50, (1979).

\bibitem{sagan} B.\ Sagan, {\em The Symmetric Group:  Representations,
Combinatorial Algorithms, and Symmetric Functions,} Springer Verlag, (2001).

\bibitem{stanley} R.\ Stanley, {\em Enumerative Combinatorics}, Vol.\
2, Cambridge, Cambridge University Press,  (2001).

\bibitem{Thomp-prod} R.C.\ Thompson, ``Smith Invariants of a Product
of Interal Matrices,'' {\em Contemp.\ Math.} 47, pp.\ 401-435,
(1985).

\bibitem{Thompson} R.C.\ Thompson, ``Interlacing Inequalities for Invariant
Factors,'' {\em Linear Algebra and its Applications}, pp.\ 1-32, (1979).

\bibitem{zel} A.\ Zelevinsky, ``\LR\ Semigroups," {\em New
Perspecties in Algebraic Combinatorics} (L.J.\ Billera, A.\
Bj\"{o}rner, C.\ Greene, R.E.\ Simion, R.P.\ Stanley, eds.)
Cambridge University Press (MSRI Publication, pp.\ 337-345 (1999).
\end{thebibliography}
\end{document}